\documentclass[11pt]{article}

\usepackage[margin=1in]{geometry} 

\IfFileExists{bbm.sty}{\usepackage{bbm}}{}
\usepackage{mathrsfs}
\usepackage{amssymb,amsmath,amsthm,amsfonts}
\usepackage{indentfirst}
\usepackage{color}
\usepackage[hang,flushmargin]{footmisc}
\usepackage{breqn}
\usepackage{graphicx}

\numberwithin{equation}{section}
\allowdisplaybreaks

\def\maslov{{\mu_{\mathrm{Maslov}}}}
\def\cz{{\mu}}

\def\vk{\vec{k}}
\def\lm{{\lambda}}

\def\sg{{\sigma}}
\def\aa{{\alpha}}

\def\om{\omega}
\def\Om{\Omega}

\def\ga{\gamma}

\newcommand{\hmu}{\hat{\mu}}

\newcommand{\R}{\mathbb{R}}
\newcommand{\Z}{\mathbb{Z}}
\newcommand{\Q}{\mathbb{Q}}

\newcommand{\N}{\mathbb{N}}

\newcommand{\U}{\mathbb{U}}

\newcommand{\Sp}{\mathrm{Sp}}
\newcommand{\Gr}{\mathrm{Gr}}
\newcommand{\TSp}{\widetilde{\mathrm{Sp}}}
\newcommand{\TSpn}{\widetilde{\mathrm{Sp}}{}^*}

\newtheorem{theorem}{Theorem}[section]

\newtheorem{proposition}[theorem]{Proposition}
\newtheorem{lemma}[theorem]{Lemma}
\newtheorem{corollary}[theorem]{Corollary}
\newtheorem{conjecture}[theorem]{Conjecture}
\newtheorem{remark}[theorem]{Remark}

\def\CP{{\mathbb C}P}
\def\RP{{\mathbb R}P}
\def\HP{{\mathbb H}P}
\def\CaP{{\text Ca}P}

\def\nn{\nonumber}

\usepackage[colorlinks, linkcolor=black, anchorcolor=black, citecolor=black, urlcolor=black, filecolor=black, menucolor=black, CJKbookmarks=true]{hyperref}

\definecolor{restoredgreen}{rgb}{0,0.42,0.16}

\makeatletter
\let\revisionoriginaltagform\tagform@
\renewcommand{\tagform@}[1]{{\color{black}\revisionoriginaltagform{#1}}}
\let\revisionoriginaleqnnum\@eqnnum
\renewcommand{\@eqnnum}{{\color{black}\revisionoriginaleqnnum}}
\makeatother

\makeatletter
\def\reflb#1#2{\begingroup
	#2%
	\def\@currentlabel{#2}%
	\phantomsection\label{#1}\endgroup
}
\makeatother

\def\HC{{\mathrm{HC}}}

\def\H{{\mathrm{H}}}

\begin{document}
\title{Multiplicity Bounds for Contractible Closed Reeb Orbits on Prequantization Bundles}

\author{
    Peng Liu\thanks{School of Mathematical Sciences, Nankai University, Tianjin, China.
        E-mail address: liupeng@nankai.edu.cn.}
}

\hypersetup{pdftitle={Multiplicity Bounds for Contractible Closed Reeb Orbits on Prequantization Bundles},pdfauthor={Peng Liu}}

\date{}
\maketitle

\begin{abstract}
We study the multiplicity of contractible closed Reeb orbits
on prequantization circle bundles $(M^{2n+1},\xi)$ over
closed symplectic manifolds $(B^{2n},\omega)$.
For even $n$, under suitable geometric hypotheses,
we prove that every non-degenerate contact form supporting
$\xi$ whose contractible closed orbits have nonzero mean
index and whose good contractible closed orbits have
nonzero Conley--Zehnder index carries at least
\[
\frac{1}{2}\chi(B)
+\frac{1}{2}\dim \H_n(B;\mathbb{Q})
+1
\]
geometrically distinct contractible closed Reeb orbits.

In dimension five, under the corresponding geometric
hypotheses and assuming $\H_1(B;\mathbb{Q})=0$,
we prove the conjectural lower bound
$\dim \H_*(B;\mathbb{Q})$
for every non-degenerate supporting contact form
whose simple contractible closed orbits have nonzero
Conley--Zehnder index.
\end{abstract}

\bigskip


\medskip


\section{Introduction}

\subsection{Introduction and Main Results}

We assume throughout that all manifolds are connected.

Let $(M^{2n+1},\xi)$ be a closed, co-oriented contact manifold. A \textit{contact structure} $\xi$ is a maximally non-integrable hyperplane field, and a \textit{contact form} $\alpha$ is a $1$-form such that $\xi = \ker \alpha$. The co-orientability of $\xi$ ensures that $\alpha$ is globally defined up to multiplication by a positive function. The associated \textit{Reeb vector field} $R_\alpha$ is uniquely determined by
\[
\iota_{R_\alpha} d\alpha = 0, \qquad \alpha(R_\alpha) = 1.
\]
The flow of $R_\alpha$ preserves the contact structure, and its integral curves are called \textit{Reeb orbits}. A Reeb orbit $\gamma: \mathbb{R} \to M$ is called \textit{closed} if it is periodic, i.e.\ if there exists $T > 0$ such that $\gamma(t+T) = \gamma(t)$ for all $t$; in this case $T$ is called its period. A closed orbit is \textit{simple} if it is not an iteration of any other closed orbit. Two closed orbits $\gamma$ and $\eta$ are said to be \textit{geometrically distinct} if their images $\gamma(\mathbb{R})$ and $\eta(\mathbb{R})$ are distinct subsets of $M$. 

It is important to distinguish between the following notions.
A contractible closed orbit is called
\textit{simple contractible} if it is not an iterate of
another contractible closed orbit.
Such an orbit may be an iterate of a non-contractible
closed orbit.

The stability type of a closed Reeb orbit $\gamma$ is encoded in its \textit{Floquet multipliers}, i.e.\ the eigenvalues of the restricted linearized Poincaré return map
\[
(d\phi_\alpha^T)_{\gamma(0)}|_{\xi}\colon \xi_{\gamma(0)} \to \xi_{\gamma(0)},
\]
where $\phi_\alpha^t$ denotes the Reeb flow and $T$ is the period of $\gamma$. We denote by $\sigma(\gamma)$ this spectrum and write $\mathbb{U} := \{ z \in \mathbb{C} : |z| = 1 \}$ for the unit circle. We classify $\gamma$ as follows:
\begin{itemize}
	\item \textit{non-degenerate} if $1 \notin \sigma(\gamma)$;
	\item \textit{hyperbolic} if $\sigma(\gamma) \cap \mathbb{U} = \emptyset$;
	\item \textit{elliptic} if $\sigma(\gamma) \subset \mathbb{U}$.
\end{itemize}
The contact form $\alpha$ is called \textit{non-degenerate} if every closed Reeb orbit, including all its iterates, is non-degenerate.

The study of closed orbits in Hamiltonian dynamics is a central topic with a long history. In the context of contact geometry, this becomes the problem of establishing multiplicity results for periodic orbits of Reeb flows on a given contact manifold. A powerful algebraic approach uses invariants such as \emph{(positive) equivariant symplectic homology}~\cite{Vit, BO09mb, BO10, BO13a, BO13b, BO17} and \emph{linearized contact homology}~\cite{EGH}. 

When these homological invariants are unbounded, that is, when the dimension of the corresponding graded vector spaces tends to infinity along some sequence of degrees $k_i \to \infty$, a celebrated result asserts that every contact form on the manifold must possess infinitely many simple closed orbits~\cite{HM, Mcl12}. We refer to this situation as the \emph{homologically unbounded} case; see the survey~\cite{Mac}. Consequently, the most interesting and technically subtle situation is the \emph{homologically bounded} case.

The \emph{prequantization circle bundles} considered below
are homologically bounded. Let $(B^{2n},\omega)$ be a closed integral symplectic manifold. Its prequantization bundle is a contact manifold $(M^{2n+1},\xi)$, where $M$ is the total space of an $S^1$-bundle over $B$ with first Chern class $- [\omega]$ (throughout the paper we adopt the convention $c_1 = -[\omega]$, rather than $[\omega]$). A prototypical example is the standard contact sphere $(S^{2n+1},\xi_{\mathrm{std}})$, which is the prequantization of $(\CP^n,\omega_{\mathrm{FS}})$.

When the symplectic base $(B,\omega)$ is aspherical (i.e.\ $\omega|_{\pi_2(B)} = 0$), it is known that under mild assumptions, every contact form on $M$ has infinitely many simple closed orbits~\cite{GGM15,GGM18}; see also~\cite{GS}. The situation is more delicate when $\omega$ is not aspherical, as there exist contact forms with only finitely many closed orbits. For instance, the standard contact sphere $S^{2n+1}$ is a prequantization of $\CP^n$ and admits contact forms with precisely $n+1$ simple closed orbits, e.g.\ irrational ellipsoids. This naturally raises the problem of finding an optimal lower bound for the number of simple closed Reeb orbits on such manifolds.

Regarding the multiplicity problem for closed Reeb orbits on prequantization $S^1$-bundles, recent research has led to the following conjecture, cf.~\cite{AbM24}:

\begin{conjecture}\label{conj:multiplicity}
	Let $(M^{2n+1},\xi)$, with $n\ge 1$, be a prequantization $S^1$-bundle over a closed symplectic manifold $(B,\omega)$, and let $\alpha$ be a contact form on $M$ that supports $\xi$. Then $\alpha$ admits at least $r_B$ simple periodic orbits, where
	\[
	r_B := \dim \H_*(B;\mathbb{Q}) \, .
	\]
\end{conjecture}

\begin{remark}
	In the three-dimensional case (i.e.\ $n=1$), Conjecture~\ref{conj:multiplicity} is known to hold. In fact, Cristofaro-Gardiner and Hutchings~\cite{CGH16} proved a sharp result: every contact form on a closed three-manifold has at least two embedded Reeb orbits. Consider a prequantization $S^1$-bundle over a closed surface $B$. When $B \cong S^2$, we have $r_B = \dim \H_*(B;\mathbb{Q}) = 2$, and this lower bound is sharp. On the other hand, if $B$ has genus $g \ge 1$ (so in particular $B \cong T^2$ when $g=1$), then by~\cite[Theorem~6.4]{GGM18} (one easily checks that in this case conditions (i) and (ii) of that theorem are automatically satisfied), it follows that every supporting contact form $\alpha$ must have infinitely many simple closed Reeb orbits.
\end{remark}

Conjecture~\ref{conj:multiplicity} remains open in higher dimensions without additional assumptions. In particular, for $n \ge 2$, without imposing any
non-degeneracy or index or action conditions, it is not even known
whether there must exist more than one simple closed Reeb orbit.
For the standard contact sphere $S^{2n+1}$, \c{C}ineli, Ginzburg, and G\"urel~\cite{CGG24} proved that every dynamically convex contact form has at least $n+1$ simple closed Reeb orbits.

Ginzburg, G\"urel, and Macarini~\cite{GGM18} proved that,
under suitable topological and filling assumptions, every
non-degenerate supporting contact form on a prequantization
circle bundle carries at least two geometrically distinct
contractible closed Reeb orbits.
They also established multiplicity results under additional
index conditions, including the requirement that the mean
Conley--Zehnder indices of all contractible closed Reeb orbits
are either all positive or all negative.
Using contact homology, they also obtained multiplicity results
for prequantization bundles over negative monotone symplectic
manifolds $B$, with the additional requirement $c_B\ge n$
\cite[Theorem~6.2]{GGM18}.

Duan, Liu, Long, and Wang \cite{DLLW24}
extended these multiplicity results to the case of
nonzero mean indices of either sign under a
symplectically aspherical filling hypothesis.
The condition $\Omega|_{\pi_2(W)}=0$ on the filling
$(W,\Omega)$ can be dropped using Novikov coefficients
\cite{AbM24}. In the non-filling setting, Ginzburg, G\"urel, and Macarini \cite{GGM18} obtained multiplicity results
using linearized contact homology.
Combining these results with arguments from their proofs, we obtain
the following theorem.

We first recall the relevant notation.
Let $(M^{2n+1},\xi)$ be a closed contact manifold satisfying 
\[
c_1(\xi)|_{\pi_2(M)} = 0 ,
\]
and let $\alpha$ be a contact form supporting $\xi$. Following~\cite{GGM18}, we say that $\alpha$ is \emph{index-positive} (respectively, \emph{index-negative}) if the mean index $\hat{\mu}(\gamma)$ is positive (respectively, negative) for every contractible periodic orbit $\gamma$ of $\alpha$; if either condition holds, $\alpha$ is said to be \emph{index-definite}. A non-degenerate contact form $\alpha$ is called \emph{index-admissible} if it has no contractible periodic orbits with Conley--Zehnder index equal to $2-n$ or $2-n \pm 1$. For convenience, we furthermore say that a contact form $\alpha$ is \emph{index-nonzero} if $\hat{\mu}(\gamma) \neq 0$ for every contractible periodic orbit $\gamma$ of $\alpha$. Below, as is customary, a non-degenerate periodic orbit $\gamma$ is said to be \emph{good} if its Conley--Zehnder index $\mu(\gamma)$ has the same parity as that of the underlying simple orbit; see Section~\ref{sec:index_orbit}. In what follows, given a symplectic manifold $B$, we denote by $\chi(B)$ its Euler characteristic and by $c_B$ its minimal Chern number, defined by
$$
c_B := 
\begin{cases}
	\displaystyle \inf \big\{ k \in \mathbb{N} \mid \exists S \in \pi_2(B) \text{ with } \langle c_1(TB), S \rangle = k \big\}, & \text{if } c_1(TB)\big|_{\pi_2(B)} \neq 0, \\[0.4em]
	0, & \text{if } c_1(TB)\big|_{\pi_2(B)} = 0.
\end{cases}
$$
Suppose that $c_B\ne0$. We say that $(B,\omega)$ is
\emph{spherically monotone} if there exists $\lambda\in\mathbb R\setminus \{0\}$
such that
\[
c_1(TB)|_{\pi_2(B)}=\lambda[\omega]|_{\pi_2(B)}.
\]
We call $(B,\omega)$ \emph{spherically positive monotone} if
$\lambda>0$ and \emph{spherically negative monotone} if $\lambda<0$.


\begin{theorem}[after \cite{GGM18,DLLW24,AbM24}]\label{prop:previous-prequantization}
Let $(M^{2n+1},\xi)$ be a prequantization $S^1$-bundle of a closed symplectic manifold $(B,\omega)$ such that $\omega|_{\pi_2(B)}\neq 0$, $c_B>n/2$, and either $\H_{k}(B;\mathbb{Q})=0$ for every odd $k$ or $c_B>n$. Let $\alpha$ be a non-degenerate contact form supporting $\xi$ that is index-nonzero and has no contractible good periodic Reeb orbits $\gamma$ with Conley--Zehnder index $\mu(\gamma)=0$ when $n$ is odd, or $\mu(\gamma) \in \{0,\pm 1\}$ when $n$ is even. Assume that one of the following conditions holds:
\begin{itemize}
\item[(F)] $M$ admits a strong symplectic filling $(W,\Omega)$ such that $c_1(TW)|_{\pi_2(W)}=0$, and the inclusion-induced map $\pi_1(M) \to \pi_1(W)$ is injective.
\item[(NF)] $c_1(\xi)|_{\H_2(M;\Q)}=0$, and $c_B\ge n$ if $B$ is spherically negative monotone.
\end{itemize}
Then $\alpha$ carries at least $\widetilde r_B$
geometrically distinct contractible closed Reeb orbits, where
\[
\widetilde r_B=
\begin{cases}
\chi(B)+2\dim \H_n(B;\mathbb Q), & n\text{ is odd},\\[2mm]
\chi(B)+4\dim \H_{n-1}(B;\mathbb Q), & n\text{ is even}.
\end{cases}
\]
\end{theorem}

\begin{remark}\label{rem:comparison-previous-am}
More recently, Abreu and Macarini~\cite{AbM26} obtained
multiplicity results 
for contact manifolds with periodic positive
$S^1$-equivariant symplectic homology. Their Theorem~1 assumes
that this homology is supported in degrees of a single parity.
To compare their result with
Theorem~\ref{prop:previous-prequantization}, suppose that $n$ is even,
that $\H_k(B;\mathbb Q)=0$ for every odd $k$, and that
$c_1(TB)=\lambda[\omega]$ in $\H^2(B;\mathbb R)$ for some $\lambda>0$.
Let $\alpha$ be a non-degenerate supporting contact form whose
contractible closed orbits all have nonzero mean index.
Under the other hypotheses of the respective results,
Theorem~\ref{prop:previous-prequantization} gives at least
$r_B=\dim\H_*(B;\mathbb Q)$ geometrically distinct contractible
closed Reeb orbits if no good contractible orbit has index $-1$,
$0$, or $1$. Abreu and Macarini obtain the same bound if, for
some positive odd integer $p$, no good contractible orbit has
index $p$ or $-p$. In particular, taking $p=1$ gives $r_B$
without a separate exclusion of index zero
(see also the comparison in \cite[Remark~8.2]{AbM26}).
Equality in their bound holds precisely when every contractible
closed Reeb orbit has even index
\cite[Theorems~1 and~2\textup{(PM)} and Remark~2.11]{AbM26}.

Abreu and Macarini assume in the statement of case \textup{(PM)}
that the odd-degree rational homology of $B$ vanishes. When $c_B>n$,
both their computation for a positive monotone base and
Theorem~\ref{prop:previous-prequantization} allow nonzero
odd-degree rational homology. For even $n$ and $p=1$, their
computation gives $\chi(B)+2\dim\H_{n-1}(B;\mathbb Q)$ for the
contractible class, using Poincar\'e duality
(see \cite[Remark~2.22 and Section~8, formula~(8.2)]{AbM26}).
With the additional exclusion of index zero for good contractible
orbits, Theorem~\ref{prop:previous-prequantization} gives
$\chi(B)+4\dim\H_{n-1}(B;\mathbb Q)$.

Case \textup{(PM)} of \cite[Theorem~2]{AbM26} assumes that
the base is positive monotone, whereas
Theorem~\ref{prop:previous-prequantization} also allows
spherically negative monotone bases.. For positive monotone bases, their result
does not impose $c_B>n/2$ as a general hypothesis. In the
non-filling case they require $c_B>1$, which already follows
from $c_B>n/2$ for the even $n\ge2$ considered here.
Their non-filling result also requires every contractible closed
orbit $\gamma$ of the given form $\alpha$ to satisfy $\mu(\gamma)>3-n$.
\end{remark}

Examples satisfying the geometric hypotheses of
Theorem~\ref{prop:previous-prequantization} include the standard
contact sphere $S^{2n+1}$ and the unit cosphere bundles $S^*N$ of
compact rank one symmetric spaces \textup{(CROSS)} with $\dim N>2$.
(For $N=S^2$ or $\RP^2$, the filling hypothesis fails, although these cases 
are known to have at least two simple closed Reeb orbits, see \cite{GGM18}, \cite{AbM24} for details.)
In all these cases, the base manifold $B$ has vanishing odd-degree rational homology (see \cite[p.~141]{Zil}), and its minimal Chern number $c_B$ and total rank $r_B$ of $\H_*(B;\Q)$ are given in Table~\ref{tab:prequantizations}, as shown in \cite{GGM18}, \cite{AbM24}.

\begin{table}[!htb]
	\centering
	\begin{tabular}{ | c | c | c | }
		\hline
		Prequantization & $r_B = \dim \H_*(B;\Q)$ & $c_B$ \\
		\hline
		\rule{0pt}{0.4cm}
		$S^{2n+1}$ & $n+1$ & $n+1$ \\
		$S^*S^{2}$ or $S^*\RP^2$ & $2$ & $2$ \\
		$S^*S^{m}$ or $S^*\RP^m$, $m>2$ even & $m$ & $m-1$ \\
		$S^*S^{m}$ or $S^*\RP^m$, $m\geq3$ odd & $m+1$ & $m-1$ \\
		$S^*\CP^{m}$, $m\geq2$ & $m(m+1)$ & $m$ \\
		$S^*\HP^m$ & $2m(m+1)$ & $2m+1$ \\
		$S^*\CaP^2$ & $24$ & $11$ \\
		\hline
	\end{tabular}
	\vskip .2cm
	\caption[Prequantizations, total Betti numbers, and minimal Chern numbers]{Some prequantizations and the corresponding $r_B$ and $c_B$.}
	\label{tab:prequantizations}
\end{table}

\begin{corollary}
	Let $(M,\xi)$ be either the standard contact sphere $S^{2n+1}$ or the unit cosphere bundle $S^*N$ of a CROSS, and let $\alpha$ be a contact form supporting $\xi$. If $\alpha$ satisfies the hypotheses of
	Theorem~\ref{prop:previous-prequantization},
	then it has at least $r_B$ geometrically distinct
	periodic orbits.
\end{corollary}

We return to Theorem~\ref{prop:previous-prequantization}.
For even $n$, what lower bound remains if good contractible orbits
of indices $\pm1$ are allowed? We retain the geometric and homological hypotheses,
the assumption that all contractible closed orbits have nonzero
mean index, and the exclusion of good contractible orbits of index
zero.

\begin{theorem}\label{thm:main}
Let $(M^{2n+1},\xi)$, where $n$ is even, be a prequantization
$S^1$-bundle of a closed symplectic manifold $(B,\omega)$
such that $\omega|_{\pi_2(B)}\neq 0$ and $c_B>n/2$.
Assume that either $\H_k(B;\mathbb Q)=0$ for every odd $k$,
or $c_B>n$. Let $\alpha$ be a non-degenerate contact form supporting~$\xi$ that is index-nonzero and has no contractible good periodic Reeb orbits $\gamma$ with Conley--Zehnder index $\mu(\gamma)=0$. Assume that one of the following conditions holds:
	\begin{itemize}
		\item[\reflb{cond:Fe}{(F)}] $M$ admits a strong symplectic filling $(W,\Omega)$ such that  $c_1(TW)|_{\pi_2(W)}=0$, and the inclusion-induced map $\pi_1(M) \to \pi_1(W)$ is injective. 
		\item[\reflb{cond:NFe}{(NF)}] $c_1(\xi)|_{\H_2(M;\Q)}=0$, and $ c_B\ge n $ if $B$ is spherically negative monotone.
	\end{itemize}
	Then $\alpha$ carries at least $\bar{r}_B$ geometrically distinct contractible periodic orbits, where
	\[
	\bar{r}_B:=\frac{1}{2}\chi(B) + \frac{1}{2}\dim \H_{n}(B;\Q)+1.
	\]
\end{theorem}

\begin{remark}\label{rem:comparison-abreu-macarini}
When the odd-degree rational homology of $B$ vanishes,
Theorem~\ref{prop:previous-prequantization} gives the conjectural
bound $r_B$. Allowing good contractible orbits of indices $\pm1$
gives the bound $\bar r_B$ in Theorem~\ref{thm:main}.
More generally, Poincar\'e duality gives
\[
	r_B-\bar r_B
	=
	\sum_{\substack{0<j<n\\j\ \mathrm{even}}}
	\dim \H_j(B;\mathbb{Q})
	+
	3\sum_{\substack{0<j<n\\j\ \mathrm{odd}}}
	\dim \H_j(B;\mathbb{Q}).
	\]
Thus, for $n=2$, the conjectural bound is recovered precisely
when $\H_1(B;\mathbb Q)=0$. For even $n\ge4$, the inequality
$\bar r_B<r_B$ is strict because $\H_2(B;\mathbb Q)\ne0$.

\cite[Theorem~5]{AbM26} excludes good orbits of index zero
when $p=0$. Unlike 
\cite[Theorem~1]{AbM26}, it does not require the positive
$S^1$-equivariant symplectic homology to be concentrated
in degrees of a single parity.
Suppose that $B$ is positive monotone and that the hypotheses
of both theorems under comparison hold.
Let $N$ denote the number of geometrically distinct contractible
closed Reeb orbits. By \cite[Remark~2.11]{AbM26}, their theorem
applies to the contractible class. Since $n$ is even and
$2c_B>n$, the computation in
\cite[Section~8, Lemma~3]{AbM26} reduces the lower bound
stated in their Theorem~5 to
$N\ge \dim\H_n(B;\mathbb Q)-\chi(B)$.
The equality statement in \cite[Theorem~5]{AbM26} would force
every contractible closed orbit to have odd index.
However, \eqref{eq:CH} gives $\HC^0_{2c_B-n}(M)\ne0$,
so there is a good contractible closed orbit of even index.
The inequality is therefore strict.

If the odd-degree rational homology of $B$ vanishes, then
$\chi(B)\ge \dim\H_n(B;\mathbb Q)+2$. Thus even the strict
inequality gives only $N\ge \dim\H_n(B;\mathbb Q)-\chi(B)+1$,
whose right-hand side is at most $-1$. In this case the estimate
at $p=0$ does not recover the positive lower bound in
Theorem~\ref{thm:main}. When $c_B>n$, Theorem~\ref{thm:main}
also allows nonzero odd-degree rational homology.

\cite[Theorem~5]{AbM26} requires the homology to vanish
in all sufficiently negative degrees. For a negative monotone base,
however, \eqref{eq:CH-nm} gives $\HC^0_{-2mc_B-n}(M)\ne0$ for
every $m\ge1$, so this hypothesis fails.
\end{remark}

Now consider either the standard contact sphere $(S^{2m-1},\xi)$, which is the prequantization $S^1$-bundle over $(\CP^{m-1},\omega)$, or the unit cosphere bundle $(S^*N,\xi)$ with $N=S^{m}$ or $\RP^{m}$, which  is the prequantization $S^1$-bundle over the Grassmannian of oriented $2$-planes $(\Gr_2^+(\R^{m+1}),\omega)$. In both cases we apply Theorem~\ref{thm:main} with $n=m-1$. Using Ziller's computation of the homology groups of the base in \cite[p.~141]{Zil}, we obtain the following corollaries.

\begin{corollary}\label{cor:main1}
	Let $(M,\xi)$ be the standard contact sphere $S^{2m-1}$, where $m$ is odd and $m\geq 3$. Let $\alpha$ be a non-degenerate contact form supporting~$\xi$. If $\alpha$ is index-nonzero and has no good periodic Reeb orbits $\gamma$ with Conley--Zehnder index $\mu(\gamma)=0$, then it has at least $\frac{m+3}{2}$ geometrically distinct periodic orbits.
\end{corollary}

\begin{corollary}\label{cor:main2}
	Let $(M,\xi)$ be the unit cosphere bundle $S^*N$ with $N=S^{m}$ or $\RP^{m}$, where $m$ is odd and $m\geq 3$. Let $\alpha$ be a non-degenerate contact form supporting~$\xi$. If $\alpha$ is index-nonzero and has no contractible good periodic Reeb orbits $\gamma$ with Conley--Zehnder index $\mu(\gamma)=0$, then it has at least $\frac{m+5}{2}$ 
	geometrically distinct periodic orbits.
\end{corollary}

\begin{remark}
	Corollary~\ref{cor:main1} gives the multiplicity bound of
	\cite[Theorem~1.1]{DuX25b} in the contact setting of the
	standard contact sphere.

	\c{C}ineli, Ginzburg, and G\"urel~\cite[Theorem~B]{CGG26} proved that a non-degenerate Reeb flow on the boundary of a star-shaped domain in $\R^{2m}$ has infinitely many simple closed orbits if it has a closed orbit with negative mean index. Thus, in Corollary~\ref{cor:main1}, if there are only finitely many simple closed orbits, the index-nonzero assumption implies that $\alpha$ is index-positive.
	
	For Corollary~\ref{cor:main2}, in general the Reeb flow of a supporting 
	contact form $\alpha$ on $S^*N$ is not a geodesic flow. However, if $\alpha$ 
	is induced by a bumpy Finsler metric $F$ on $N$ (so that $S^*N$ is the unit 
	cosphere bundle of $(N,F)$ and $\alpha$ is the restriction of the Liouville 
	form), then the Reeb flow is a reparametrized geodesic flow and the 
	Conley--Zehnder index of a closed Reeb orbit coincides with the Morse index of 
	the corresponding closed geodesic (see \cite{Liu05} and 
	Sections~\ref{sec:index_orbit},~\ref{sec:IIF} below). In particular, Bott's 
	iteration formula (see, for instance, \cite[Theorem~9.2.1]{Long02}) implies 
	that, in this Finsler situation, it suffices to assume that $\alpha$ has no 
	simple contractible periodic orbits $\gamma$ with $\mu(\gamma)=0$ in order to 
	ensure that there is no contractible good orbit with $\mu(\gamma)=0$ and that 
	$\alpha$ is index-positive (and hence index-nonzero) in our sense. Consequently, Corollary~\ref{cor:main2} contains \cite[Theorem~1.1]{DuX25a} 
	(stated for Finsler geodesic flows on $(S^{m},F)$) as a special case.
\end{remark}

In dimension five ($n=2$), Remark~\ref{rem:comparison-abreu-macarini}
gives $\bar r_B=r_B$ whenever $\H_1(B;\mathbb Q)=0$.
Suppose that the contact form has only finitely many geometrically
distinct contractible closed orbits. Lemmas~\ref{lemma:nonzeroindex}
and~\ref{lemma:nonzeromean} then show that excluding simple
contractible orbits of index zero suffices for both index assumptions
of Theorem~\ref{thm:main}. This gives the conjectural lower bound in
 dimension five under the following hypotheses.

\begin{theorem}\label{thm:minor}
	Let $(M^{5},\xi)$ be a prequantization $S^1$-bundle of a closed symplectic manifold $(B,\omega)$ such that $\omega|_{\pi_2(B)}\neq 0$, $c_B>1$, and $\H_{1}(B;\mathbb{Q})=0$. Let $\alpha$ be a non-degenerate contact form supporting $\xi$ that has no simple contractible periodic Reeb orbits $\gamma$ with Conley--Zehnder index $\mu(\gamma)=0$. Assume that one of the following conditions holds:
	\begin{itemize}
		\item[\reflb{cond:F0}{(F)}] $M$ admits a strong symplectic filling $(W,\Omega)$ such that $c_1(TW)|_{\pi_2(W)}=0$, and the inclusion-induced map $\pi_1(M) \to \pi_1(W)$ is injective. 
		\item[\reflb{cond:NF0}{(NF)}] $c_1(\xi)|_{\H_2(M;\Q)}=0$.
	\end{itemize}
	Then $\alpha$ carries at least $r_B$ geometrically distinct contractible periodic orbits.
\end{theorem}

\begin{remark}
	For the standard contact sphere $S^5$,
	Theorem~\ref{thm:minor} recovers the result of
	Duan, Liu, Long, Qi, and Wang~\cite[Theorem~1.1(ii)]{DLLQW25}
	for non-degenerate star-shaped hypersurfaces in $\R^6$.
	Their condition excludes simple closed characteristics
	with Maslov-type index $-1$, which corresponds to
	Conley--Zehnder index $0$ in our convention.
	
\end{remark}

\subsection{Organization and notation}

\subsubsection{Organization}

In Section~2 we collect the necessary preliminaries, including a brief review of the Conley--Zehnder index for symplectic paths and periodic Reeb orbits, the positive \(S^1\)-equivariant symplectic homology of prequantization bundles, as well as local equivariant symplectic homology, the resonance relation, and Morse-type inequalities. Section~3 is devoted to index iteration theory for symplectic paths; in particular, we recall an index recurrence theorem that will be a key ingredient in our multiplicity arguments. In Section~4 we apply these tools to prove our main multiplicity result for contractible closed Reeb orbits on prequantization \(S^1\)-bundles. Finally, in Section~5 we prove Theorem~\ref{thm:minor}.

\subsubsection{Conventions and notation}\label{subsec:notation}

Throughout this paper, let $\mathbb{N}$, $\mathbb{Z}$, $\mathbb{Q}$ and $\mathbb{R}$ denote the sets of positive integers, integers, rational numbers and real numbers, respectively. 
For two symplectic matrices in block form
\[
M_k = \begin{pmatrix}
A_k & B_k \\
C_k & D_k
\end{pmatrix} \in \mathrm{Sp}(2m_k), \quad k=1,2,
\]
their $\diamond$-product $M_1 \diamond M_2 \in \mathrm{Sp}(2m_1 + 2m_2)$ is defined by
\[
M_1 \diamond M_2 = \begin{pmatrix}
A_1 & 0 & B_1 & 0 \\
0 & A_2 & 0 & B_2 \\
C_1 & 0 & D_1 & 0 \\
0 & C_2 & 0 & D_2
\end{pmatrix},
\]
which represents the \emph{symplectic direct sum} in the canonical coordinates on $\mathbb{R}^{2(m_1+m_2)}$.

\section{Preliminaries}

\subsection{The Conley--Zehnder index for paths of symplectic matrices}
\label{sec:index_paths}

We recall the definition and properties of the Conley--Zehnder index, following the conventions and notation of \cite[Section~ 4]{GG20}. For a detailed treatment, we refer to \cite{Long02,SZ92}.

Let $\Phi\colon [0,1]\to\Sp(2n)$ be a continuous path with $\Phi(0)=I$. The \emph{mean index} $\hmu(\Phi)\in\mathbb{R}$ is a homotopy invariant of the path with fixed endpoints. It measures the total rotation angle of certain eigenvalues of $\Phi(t)$ on the unit circle. For a homotopy family of paths $\Phi_s$, we have $\hmu(\Phi_s)\equiv\text{constant}$ provided the eigenvalues of $\Phi_s(1)$ remain constant.

We identify such paths, up to homotopy with fixed endpoints,
with elements of the universal covering group $\TSp(2n)$. The map $\hmu\colon \widetilde{\Sp}(2n)\to\mathbb{R}$ is the unique continuous homogeneous quasimorphism on the universal covering group $\TSp(2n)$ satisfying the normalization condition
\[
\hmu(\Phi_0)=2\quad \text{for}\quad \Phi_0(t)= \exp(2\pi i t)\diamond I_{2n-2},\quad t\in [0,1],
\]
as established in \cite{BG92}. Here homogeneity means $\hmu(\Phi^k)=k\hmu(\Phi)$ for all $k\in\mathbb{N}$, and the quasimorphism property asserts that
\begin{equation*}
|\hmu(\Phi\cdot\Psi)-\hmu(\Phi)-\hmu(\Psi)|\leq C_n,
\end{equation*}
where  $C_n$ depends only on $n$ (and may be taken as $C_n=4n$ according to \cite{Ust99}).

A path $\Phi\in\TSp(2n)$ is called \emph{non-degenerate} if $\det(\Phi(1)-I)\neq 0$, i.e., if no eigenvalue of the endpoint $A=\Phi(1)$ equals 1. Let $\Sp^*(2n)$ denote the set of such matrices, and $\TSpn(2n)$ the part of $\TSp(2n)$ lying over
$\Sp^*(2n)$. For a non-degenerate path $\Phi$, the \emph{Conley--Zehnder index} $\mu(\Phi)\in\mathbb{Z}$ is defined as follows: connect $A=\Phi(1)$ to a symplectic transformation with elliptic part $-I$ (when non-trivial) by a path $\Psi$ contained in $\Sp^*(2n)$ (it is not hard to see that such a path exists), and set $\mu(\Phi)=\hmu(\Phi\ast\Psi)$, where $\Phi\ast\Psi$ denotes the concatenation. One can show that this definition is independent of the choice of $\Psi$ and the function $\mu\colon \TSpn(2n)\to\mathbb{Z}$ is constant
on connected components of $ \TSpn(2n) $, namely, it is locally constant. A path $\Phi\in\TSp(2n)$ is \emph{strongly non-degenerate} if $\Phi^k$ is non-degenerate for each $ k\in\N $, i.e., if no eigenvalue of $\Phi(1)$ is a root of unity.

The following properties are established in, e.g., \cite{Long02,SZ92}:
\begin{enumerate}
	\item For any non-degenerate path $\Phi$, $\mu(\Phi^{-1}) = -\mu(\Phi)$.
	\item Both indices are additive under direct sums: for $\Phi\in\TSp(2n)$ and $\Psi\in\TSp(2n')$,
	\[
	\hmu(\Phi\diamond\Psi) = \hmu(\Phi) + \hmu(\Psi), \quad 
	\mu(\Phi\diamond\Psi) = \mu(\Phi) + \mu(\Psi),
	\]
	where the second equality requires both paths to be non-degenerate.
	\item The indices satisfy the inequality
	\begin{equation}
	\label{eq:range}
	|\hmu(\Phi) - \mu(\Phi)| < n \quad \text{for all } \Phi\in\TSpn(2n),
	\end{equation}
	and consequently, for strongly non-degenerate $ \Phi $ 
	\begin{equation*}
	\lim_{k\to\infty} \frac{\mu(\Phi^k)}{k} = \hmu(\Phi),
	\end{equation*}
	which justifies the terminology ``mean index'' for $\hmu$.
\end{enumerate}

\subsection{The Conley--Zehnder index of periodic orbits}
\label{sec:index_orbit}

Let $\gamma$ be a strongly non-degenerate closed Reeb orbit of period $T>0$. Parametrize it on $S^1=\mathbb R/\mathbb Z$ by $\gamma(t)=\phi_\alpha^{Tt}(\gamma(0))$, and let $\Psi\colon \gamma^*\xi \to S^1 \times \mathbb{R}^{2n}$ be a symplectic trivialization of the contact structure $\xi$ along $\gamma$. For each $t$, define $\Psi_t\colon \xi(\gamma(t)) \to \mathbb{R}^{2n}$ as the composition of $\Psi|_{\gamma^*\xi(t)}$ with projection onto the second factor. Through this trivialization, the linearized Reeb flow induces a path of symplectic matrices
\[
\Phi(t) = \Psi_t \circ d\phi_\alpha^{Tt}(\gamma(0))|_\xi \circ \Psi_0^{-1},\qquad 0\leq t\leq1,
\]
where $\phi^t_\alpha$ denotes the Reeb flow of $\alpha$. The Conley--Zehnder index and mean index of $\gamma$ with respect to the trivialization $\Psi$ are then defined as
\[
\mu(\gamma;\Psi) = \mu(\Phi) \quad \text{and} \quad \hmu(\gamma;\Psi) = \hmu(\Phi),
\]
respectively. These indices depend only on the homotopy class of the trivialization $\Psi$. Specifically, if $\Upsilon$ is another trivialization of $\gamma^*\xi$, then
\[
\mu(\gamma;\Upsilon) = \mu(\gamma;\Psi) + 2\maslov(\Upsilon_t \circ \Psi_t^{-1}),
\]
where $\maslov$ denotes the Maslov index, which provides an isomorphism between $\pi_1(\Sp(2n))$ and $\mathbb{Z}$. In particular, the parity of the Conley--Zehnder index is independent of the choice of trivialization. Moreover, the Conley--Zehnder indices of all even iterates have the same parity, and the same holds for all odd iterates: for all $j,\,k \in \N$,
\begin{equation}\label{same parity}
\cz(\ga^{2j};\Psi^{2j}) \equiv \cz(\ga^{2k};\Psi^{2k})\ \text{and}\
\cz(\ga^{2j-1};\Psi^{2j-1}) \equiv \cz(\ga^{2k-1};\Psi^{2k-1})\
\pmod{2}. 
\end{equation}

A periodic orbit $\gamma$ is called \emph{good} if its Conley--Zehnder index has the same parity as that of the underlying simple orbit. Otherwise, $\gamma$ is called \emph{bad}. As noted above, this classification is independent of the choice of trivialization.

When $\gamma$ is contractible, there is a canonical way to choose a trivialization up to homotopy. Let $\sigma\colon D^2 \to M$ be a capping disk for $\gamma$, i.e., a smooth map from the 2-disk $D^2$ to $M$ such that $\sigma|_{\partial D^2} = \gamma$. Choose a trivialization of the pullback bundle $\sigma^*\xi$ over $D^2$, and let $\Psi$ be its restriction to the boundary $\partial D^2 \cong S^1$. Since $D^2$ is contractible, the homotopy class of $\Psi$ is independent of the choice of trivialization of $\sigma^*\xi$. Furthermore, if
$$c_1(\xi)|_{\pi_2(M)}=0,$$
 then the homotopy class of $\Psi$ is also independent of the choice of capping disk $\sigma$. 

Throughout this paper, for a contractible orbit $\gamma$, we denote by $\mu(\gamma)$ and $\hmu(\gamma)$ the Conley--Zehnder index and mean index computed with respect to this standard trivialization.

\subsection{Equivariant Symplectic Homology of Prequantizations}\label{sec:esh_preq}

We use positive $S^1$-equivariant symplectic homology
for contractible periodic orbits, graded by the
Conley--Zehnder index; see
\cite[Section~3.3]{GGM18} and
\cite[Proposition~3.3]{GG20}.
Suppose that $(M,\xi)$ admits a strong filling $(W,\Omega)$
with $\Omega|_{\pi_2(W)}=0$ and $c_1(TW)|_{\pi_2(W)}=0$.
For a non-degenerate contact form $\alpha$ supporting $\xi$,
the positive $S^1$-equivariant symplectic homology of $W$
is computed by a chain complex generated by the good closed
Reeb orbits of $\alpha$.
This complex decomposes as a direct sum of subcomplexes
indexed by the free homotopy classes of loops in $W$.
Each subcomplex is generated by the orbits in the corresponding
class.
If $\pi_1(M)\to\pi_1(W)$ is injective, an orbit in $M$ is
contractible in $W$ if and only if it is contractible in $M$.
We denote the corresponding subcomplex by
$\mathcal C_*^0(\alpha)$ and its homology by
\[
\HC_*^0(M):=\H_*(\mathcal C_*^0(\alpha)).
\]
This complex is defined over $\mathbb Q$, graded by the Conley--Zehnder index computed using capping disks, and
filtered by action. Its differential depends on auxiliary choices, but its homology
is independent of these choices and may a priori depend on the filling.

In the non-filling case, assume that
$c_1(\xi)|_{\pi_2(M)}=0$ and that every contractible closed
Reeb orbit of $\alpha$ has index greater than $3-n$.
The construction in \cite[Section~4.1.2]{BO17} then defines
positive $S^1$-equivariant symplectic homology for periodic
Reeb orbits contractible in $M$.
As in \cite[Section~3.3]{GGM18}, we again denote this homology
by $\HC_*^0(M)$.
The required index bound follows if $\alpha$ is both
index-positive and index-admissible, by \eqref{eq:range}.
The Novikov-coefficient and linearized contact homology
constructions used under the broader hypotheses of
Theorem~\ref{thm:main} are described at the beginning of
Section~\ref{sec:proof-main}.

The following proposition computes the equivariant symplectic homology of the prequantization \(S^1\)-bundles considered below in terms of the homology of the base; this computation will be used throughout the paper.

\begin{proposition}[{\cite[Proposition 3.1]{GGM18}}]\label{prop:CH}
	Let $(M^{2n+1},\xi)$ be a prequantization of a closed symplectic manifold $(B,\omega)$ such that $\omega|_{\pi_2(B)} \neq 0$ and either $\H_{k}(B;\mathbb{Q})=0$ for every odd $k$ or $c_B>n$.
	
	\begin{itemize}
		\item[{\rm (a)}] Suppose that $M$ admits a strong symplectic filling $(W,\Omega)$ satisfying $$\Omega|_{\pi_2(W)}=0=c_1(TW)|_{\pi_2(W)},$$ and such that the inclusion-induced map $\pi_1(M) \to \pi_1(W)$ is injective. Then $B$ is spherically monotone.
		When $B$ is spherically positive monotone, the positive equivariant symplectic homology for contractible periodic orbits of $M$ is given by
		\begin{equation}
		\label{eq:CH}
		\HC^0_\ast(M) \cong \bigoplus_{m\in\mathbb{N}} \H_{\ast -2mc_B+n} (B; \mathbb{Q}).
		\end{equation}
		When $B$ is spherically negative monotone, we have
		
		\begin{equation}
		\label{eq:CH-nm}
		\HC^0_\ast(M) \cong \bigoplus_{m\in\mathbb{N}} \H_{\ast +2mc_B+n} (B; \mathbb{Q}).
		\end{equation}
		
		In particular, in both cases the homology is independent of the choice of the filling $W$ satisfying the above conditions.
		
		\item[{\rm (b)}] Alternatively, assume that $B$ is spherically positive monotone with $c_B\geq 2$, $c_1(\xi)=0$, and $\alpha$ is a non-degenerate contact form on $(M,\xi)$ such that all contractible closed Reeb orbits have index greater than $3-n$. Then \eqref{eq:CH} also holds.
	\end{itemize}
\end{proposition}

\begin{remark}\label{rm:negative-grading}
Suppose that $B$ is spherically negative monotone. The $m$th family of contractible Reeb orbits of the connection form has orbit space $B_m\cong B$ and Robbin--Salamon index $\mu_{\mathrm{RS}}(B_m)=-2mc_B$. By \cite[Lemma~2.4 and Section~9.1]{Bou02}, the orbit corresponding to a critical point $p$ of a Morse function $f$ on $B$ satisfies
\begin{equation}\label{eq:negative-fiber-index}
\mu(\gamma_{m,p})
=\mu_{\mathrm{RS}}(B_m)-\tfrac12\dim B+\operatorname{ind}_f(p)
=-2mc_B-n+\operatorname{ind}_f(p).
\end{equation}
By \eqref{eq:negative-fiber-index}, an orbit of index $k$ corresponds to a critical point of Morse index $k+2mc_B+n$. Thus the shift in \eqref{eq:CH-nm} is $+n$, rather than
the $-n$ in \cite[Proposition~3.1]{GGM18}.
\end{remark}

\begin{remark}\label{rm:chern-zero-index}
Suppose that $n$ is even and $c_B>0$.
Assume that the homology of the complex generated by the
good contractible closed Reeb orbits of $\alpha$, graded
by the Conley--Zehnder index, is computed by
\eqref{eq:CH} or \eqref{eq:CH-nm}, according to the sign
of the monotonicity constant.
Then the absence of good contractible closed Reeb orbits
of index zero forces $c_B>n/2$.

Indeed, suppose that $0<c_B\le n/2$.
In the positive monotone case, \eqref{eq:CH} gives
\[
    \HC_0^0(M)
    \cong
    \bigoplus_{m\ge1}\H_{n-2mc_B}(B;\mathbb Q),
\]
whereas in the negative monotone case, \eqref{eq:CH-nm} gives
\[
    \HC_0^0(M)
    \cong
    \bigoplus_{m\ge1}\H_{n+2mc_B}(B;\mathbb Q).
\]
In either case, the summand with $m=1$ is nonzero.
To see this, note that $n-2c_B$ and $n+2c_B$ are even
integers in $[0,2n]$, and that every even-degree rational
homology group of $B$ in this range is nonzero.
The latter follows from $[\omega]^j\ne0$ in
$\H^{2j}(B;\mathbb R)$ for $0\le j\le n$.
Thus $\HC_0^0(M)\ne0$.
On the other hand, the absence of good contractible orbits
of index zero implies that the degree-zero chain group
vanishes, so $\HC_0^0(M)=0$, a contradiction.
The same argument applies with coefficients in the
Novikov field used under condition~\ref{cond:Fe}.

This is why we state the condition $c_B>n/2$ explicitly
in Theorem~\ref{thm:main}.
The obstruction from nonzero degree-zero homology is also
discussed in \cite[Remark~2.8]{AbM26}.
\end{remark}

\subsection{Local Equivariant Symplectic Homology, Resonance Relation
 and Morse Inequalities}\label{sec:lesh}

For an isolated contractible closed Reeb orbit $\gamma$, let $\HC_*(\gamma)$ denote its local equivariant symplectic homology. If $\gamma$ is non-degenerate, then
\begin{equation}\label{lesh:nd}
\HC_*(\gamma) =
\begin{cases}
\Q, & \text{if $\gamma$ is good and $*=\mu(\gamma)$}, \\
0, & \text{otherwise}.
\end{cases}
\end{equation}
The Euler characteristic of $\gamma$ is defined by
\[
\chi(\gamma) = \sum_{m \in \mathbb{Z}} (-1)^m \dim \HC_m(\gamma).
\]
This sum is finite. When $\gamma$ is non-degenerate,
\begin{equation*}
\chi(\gamma) =
\begin{cases}
(-1)^{\mu(\gamma)}, & \text{if $\gamma$ is good}, \\
0, & \text{if $\gamma$ is bad}.
\end{cases}
\end{equation*}
Assume that every iterate of $\gamma$ is isolated.
The local \emph{mean} Euler characteristic of $\gamma$ is
\[
\hat{\chi}(\gamma) = \lim_{j\to\infty} \frac{1}{j} \sum_{k=1}^j \chi(\gamma^k).
\]
The limit above exists and is rational; see \cite{GGo}. When $\gamma$ is strongly non-degenerate, we have
\begin{equation*}
\hat{\chi}(\gamma)=
\begin{cases}
0, & \text{if $\gamma$ is bad}, \\
(-1)^{\cz(\gamma)}, & \text{if $\gamma$ and $\gamma^2$ are good}, \\
(-1)^{\cz(\gamma)}/2, & \text{if $\gamma$ is good and $\gamma^2$ is bad}.
\end{cases}
\end{equation*}

A simple contractible orbit may be a bad even iterate of a non-contractible simple orbit. In that case all its iterates are bad, so $\hat\chi(\gamma)=0$.

Assume now that $\alpha$ is index-nonzero and has finitely many distinct simple contractible closed orbits $\gamma_1,\dots,\gamma_r$.
This assumption ensures that the positive/negative mean Euler characteristic of $M$,
\begin{equation*}
\chi_\pm(M) := \lim_{j\to\infty} \frac{1}{j} \sum_{m=0}^j (-1)^m b_{\pm m},
\end{equation*}
is well-defined, where $b_m := \dim \HC^0_m(M)$ is the $m$-th Betti number; see \cite{GGo}. The mean Euler characteristic is related to local equivariant symplectic homology via the resonance relation
\begin{equation}
\label{eq:resonance}
\begin{cases}
\displaystyle
\sum_{\substack{1\leq i\leq r \\ \hat{\mu}(\gamma_i)>0}} \frac{\hat{\chi}(\gamma_i)}{\hat{\mu}(\gamma_i)} = \chi_{+}(M),\\
\displaystyle
\sum_{\substack{1\leq i\leq r \\ \hat{\mu}(\gamma_i)<0}} \frac{\hat{\chi}(\gamma_i)}{|\hat{\mu}(\gamma_i)|} = \chi_{-}(M),
\end{cases}
\end{equation}
proved in \cite{GK,HM}.

Let 
\[
c_m := \sum_{i=1}^r \sum_{k=1}^\infty \dim \HC_m(\ga_i^k)
\]
be the $m$-th Morse-type number. When $\alpha$ is non-degenerate, $c_m$ is simply the number of good orbits of index $m$. 

For any $m \in \mathbb{Z}$ and any two integers $m_1$ and $m_2$ with $m_1 \le m_2$ and the same parity, the following Morse inequalities hold:
\begin{eqnarray}
c_m &\geq& b_m,\label{eq:Morse_inequalities1}\\
c_{m_1} - c_{m_1+1} + \dots + c_{m_2} &\geq& b_{m_1} - b_{m_1+1} + \dots + b_{m_2}.\label{eq:Morse_inequalities2}
\end{eqnarray}

\section{Index Iteration Theory for Symplectic Paths}

\subsection{Index Iteration Formula}\label{sec:IIF}
In 1999, Y. Long established the basic normal form decomposition of symplectic matrices \cite{Long99}. Building on this foundation, he further developed precise iteration formulae for the indices of symplectic paths in 2000 \cite{Long00}. 

As in \cite{Long00}, denote by 
\begin{align}
D(\lm) &= \begin{pmatrix} \lm & 0 \\ 0 & \lm^{-1} \end{pmatrix}, \quad \text{for } \lm\in\R\setminus\{0, \pm 1\}, \nn \\
R(\theta) &= \begin{pmatrix} \cos\theta & -\sin\theta \\ \sin\theta & \cos\theta \end{pmatrix}, \quad \text{for } \theta\in (0,\pi)\cup (\pi,2\pi), \nn \\
N_2(e^{\theta\sqrt{-1}}, B) &= \begin{pmatrix} R(\theta) & B \\ 0 & R(\theta) \end{pmatrix}, \quad \text{for } \theta\in (0,\pi)\cup (\pi,2\pi) \text{ and} \nonumber \\
&\qquad B=\begin{pmatrix} b_1 & b_2 \\ b_3 & b_4 \end{pmatrix} \text{ with } b_j\in\R, \text{ and } b_2\neq b_3. \nn
\end{align}

For every $M\in\Sp(2n)$, the homotopy set $\Omega(M)$ of $M$ in $\Sp(2n)$ is defined by
$$ \Om(M)=\{N\in\Sp(2n)\,|\,\sg(N)\cap\U=\sg(M)\cap\U\equiv\Gamma\;\mbox{and}
\;\nu_{\om}(N)=\nu_{\om}(M),\ \forall\om\in\Gamma\}, $$
where $\sg(M)$ denotes the spectrum of $M$,  $\mathbb{U}$ denotes the unit circle in $\mathbb{C}$, and 
\[
\nu_{\omega}(M) \equiv \dim_{\mathbb{C}} \ker_{\mathbb{C}}(M - \omega I) \quad \text{for } \omega \in \mathbb{U}.
\]
We denote by $\Om^0(M)$ the path-connected component of $\Om(M)$ containing $M$.

In \cite[Theorems 1.2 and 1.3]{Long00}, Maslov-type index formulae are provided for iterations of symplectic paths (for the definition of the Maslov-type index, see, e.g., \cite{Long00} and \cite{Long02}), where the iteration is defined as concatenation. Concretely, the $k$-th iteration of a symplectic path $\Phi$ is defined by
\begin{equation*}
\Phi^{*k}(t) = \Phi(kt-j)\Phi(1)^j,
\quad \text{for } \frac{j}{k} \le t \le \frac{j+1}{k},\; j = 0,1,\dots,k-1.
\end{equation*}
In other words, $\Phi^{*k}$ is the concatenation
$
\Phi * (\Phi\Phi(1)) * \cdots * (\Phi\Phi(1)^{k-1}).
$
As observed in \cite[Remark~3.1]{Li25}, the pointwise matrix-power iteration $\Phi^k$ and the concatenation iteration $\Phi^{*k}$ are homotopic with fixed endpoints; namely,
\[
\Phi^k \sim \Phi * (\Phi\Phi(1)) * \cdots * (\Phi\Phi(1)^{k-1}).
\]
Therefore, they share the same mean index and the same Conley--Zehnder index when $ \Phi^k $ is non-degenerate.

Note that the Conley--Zehnder index coincides with the Maslov-type index for non-degenerate symplectic paths. Their relationship can be established through the Robbin--Salamon index.
We first recall that the Robbin--Salamon index $\mu_{\mathrm{RS}}$ is homotopy invariant and satisfies $\mu_{\mathrm{RS}}(\Phi) = \mu(\Phi)$ when $\Phi \in \TSpn(2n)$ (see \cite{RS1993}).
Following the notation of \cite{Long02}, for any path $\Phi \in \TSp(2n)$, we denote its Maslov-type index by the pair $(i(\Phi,1), \nu(\Phi,1))$, where $\nu(\Phi,1) = 0$ if $\Phi \in \TSpn(2n)$. According to \cite[Theorem 3.1 and Corollary 2.1]{LZ2000}, for a non-degenerate path $\Phi \in \TSpn(2n)$, the following identity holds:
\begin{equation*}
\begin{split}
\mu_{\mathrm{RS}}(\Phi) &= \mu_{\mathrm{RS}}\big(\operatorname{Gr}(\Phi),\operatorname{Gr}(I_{2n})\big) \\
&=\mu_{\mathrm{CLM}}\big(\operatorname{Gr}(I_{2n}),\operatorname{Gr}(\Phi)\big)+\tfrac{1}{2}\big(\nu_1(\Phi(1))-\nu_1(\Phi(0))\big)\\
&=\mu_{\mathrm{CLM}}\big(\operatorname{Gr}(I_{2n}),\operatorname{Gr}(\Phi)\big) + \tfrac{1}{2}\big(\nu(\Phi,1)-2n\big) \\
&= i(\Phi,1).
\end{split}
\end{equation*}
Here, $\mu_{\mathrm{CLM}}$ denotes the Cappell--Lee--Miller index \cite{CLM}, and $\operatorname{Gr}(M)$ is the graph of a symplectic matrix $M \in \Sp(2n)$, regarded as a Lagrangian subspace of $(\mathbb R^{2n}\oplus\mathbb R^{2n},-\omega_0\oplus\omega_0)$. 

The fixed-endpoint homotopy between the two iterations and the equality of the two indices for non-degenerate symplectic paths allow us to apply \cite[Theorems~1.2 and~1.3]{Long00} to $\Phi^k$. This gives the following normal form decomposition and Conley--Zehnder index iteration formula for strongly non-degenerate symplectic paths.

\begin{theorem}\label{indexformula}
	Let $\Phi \in \TSpn(2n)$ with $P = \Phi(1)$, and suppose that $\Phi^k \in \TSpn(2n)$ for all $k \in \mathbb{N}$. Then there exists a continuous path $f:[0,1]\to\Omega^0(P)$ with $f(0)=P$ and
	\begin{equation*}
	\begin{split}
	f(1) &= R(\theta_1) \diamond \cdots \diamond R(\theta_r) \diamond N_2\big(e^{\aa_{1}\sqrt{-1}},B_{1}\big) \diamond \cdots \diamond N_2\big(e^{\aa_{r_{\ast}}\sqrt{-1}},B_{r_{\ast}}\big) \diamond M_h,
	\end{split}
	\end{equation*}
	where $\frac{\theta_{j}}{2\pi}\in(0,1)\setminus\mathbb{Q}$ for $1\le j\le r$, $\frac{\alpha_{j}}{2\pi}\in(0,1)\setminus\mathbb{Q}$ for $1\le j\le r_{\ast}$, and the nonnegative integers $r$, $r_\ast$, $h$ satisfy
	\begin{equation*}
	r + 2r_{\ast} + h = n.
	\end{equation*}
	For $h\geq1$, the hyperbolic block is
	$M_h=D(2)^{\diamond h}$ or
	$M_h=D(-2)\diamond D(2)^{\diamond(h-1)}$;
	when $h=0$, this block is omitted.
	Moreover, for every $k\in\mathbb N$, the Conley--Zehnder index of the matrix-power iterate satisfies
	\begin{equation}\label{eq:CZ_iteration}
	\cz(\Phi^k) = k\big(\cz(\Phi)-r\big) + 2\sum_{j=1}^r\left\lfloor\frac{k\theta_j}{2\pi}\right\rfloor + r.
	\end{equation}
\end{theorem}

By \cite[Theorems 8.1.6--8.1.7 and Theorems 8.2.1--8.2.4]{Long02}, we have the following result on the parity of the Conley--Zehnder index.

\begin{proposition}\label{index:odd-even}
	Let $\Phi \in \TSpn(2)$ be a symplectic path.
	\begin{enumerate}
		\item If the end matrix $\Phi(1)$ is homotopic within $\Sp^*(2)$ to $R(\theta)$ for some $\theta \in (0,\pi) \cup (\pi,2\pi)$, or to $D(-2)$, then the Conley--Zehnder index $\cz(\Phi)$ is odd.
		\item If the end matrix $\Phi(1)$ is homotopic within $\Sp^*(2)$ to $D(2)$, then the Conley--Zehnder index $\cz(\Phi)$ is even.
	\end{enumerate}
	
	Let $\psi \in \TSpn(4)$ be a symplectic path. If the end matrix $\psi(1)$ is homotopic within $\Sp^*(4)$ to $N_2(e^{\alpha\sqrt{-1}},B)$ for some $\alpha \in (0,\pi) \cup (\pi,2\pi)$, then the Conley--Zehnder index $\cz(\psi)$ is even.
\end{proposition}

\subsection{Index Recurrence}

A crucial ingredient for distinguishing simple and iterated orbits in the proof of Theorem~\ref{thm:main} is the following combinatorial result on index iteration behavior. We state this as the Index Recurrence Theorem for the non-degenerate case, which is essentially \cite[Theorem 5.1]{GG20} with a change in notation.

This result represents the non-degenerate specialization of the Generalized Common Index Jump Theorem \cite[Theorem 3.6]{DLLW24}, which generalizes earlier work by Long and Zhu \cite[Theorem 4.3]{LZ02} and its enhancement \cite[Theorem 3.5]{DLW16}. According to the discussion in \cite{Duan25}, the Generalized Common Index Jump Theorem is consistent with the Index Recurrence Theorem in the degenerate case \cite[Theorem 5.2]{GG20}; we focus here solely on the non-degenerate setting.

\begin{theorem}\label{indexrecurrence}
	Let $\Phi_1,\ldots,\Phi_r$ be a finite collection of strongly non-degenerate elements of $ \widetilde{\Sp}(2n) $ such that $\hmu(\Phi_i)\neq 0$ for all $1\le i\le r$. Then for any $\eta>0$ and any $\ell_0\in\N$, there exist an integer sequence $d_j\to +\infty$ and a sequence of integer vectors $\vk_j=\left(k_{1j},\ldots,k_{rj}\right)$ with all components going to infinity as $j\to +\infty$, such that for all $i$, $j$, and for every $\ell\in\Z$ with $1\leq |\ell|\leq \ell_0$, the following hold:
	\begin{enumerate}
		\item[\textnormal{\reflb{cond:i}{(i)}}] $\big|\hmu\big(\Phi^{k_{ij}}_i\big)-\rho_{i}d_j\big|<\eta$,
		\item[\textnormal{\reflb{cond:ii}{(ii)}}] $\cz\big(\Phi^{k_{ij}+\ell}_i\big)= \rho_{i}d_j + \cz(\Phi^\ell_i)$,
	\end{enumerate}
	where
	\[
	\rho_{i}=
	\begin{cases}
	1, & \text{if }\hmu(\Phi_i)> 0, \\
	-1, & \text{if }\hmu(\Phi_i)< 0.
	\end{cases}
	\]
	Furthermore, for any $N\in \N$ we can make all $d_j$ and $k_{ij}$ divisible by $N$.
\end{theorem}

\section{Proof of Theorem~\ref{thm:main}}
\label{sec:proof-main}

We prove Theorem~\ref{thm:main'} below and first explain how the same argument applies under the hypotheses of Theorem~\ref{thm:main}.

Under condition~\ref{cond:Fe} of Theorem~\ref{thm:main}, we use positive $S^1$-equivariant symplectic homology with coefficients in the Novikov field, as in \cite[Section~2 and Proposition~2.1(a)]{AbM24}. The formulas \eqref{eq:CH} and \eqref{eq:CH-nm} then hold with $\mathbb Q$ replaced by this field; see also Remark~\ref{rm:negative-grading}. All chain and homology ranks in this case are taken over the Novikov field.

Under condition~\ref{cond:NFe} of Theorem~\ref{thm:main},
$c_1(\xi)$ is torsion, and $(B,\omega)$ is monotone;
see \cite[Remark~1.11]{AbM24}.
Using the trivializations from the proof of
\cite[Proposition~2.1]{AbM24}, we apply the arguments of
\cite[Lemma~3.3 and Section~6.2]{GGM18} to obtain a sufficiently
small non-degenerate perturbation $\beta$ of the connection
form such that every contractible orbit $\gamma$ of $\beta$
satisfies
\begin{equation}\label{eq:beta-indices}
\begin{cases}
\mu(\gamma)\geq 2c_B-n, & \text{if $B$ is spherically positive monotone},\\
\mu(\gamma)\leq n-2c_B, & \text{if $B$ is spherically negative monotone}.
\end{cases}
\end{equation}
Since $c_B\geq2$ in the positive case and $c_B\geq n$ in
the negative case, these estimates exclude the indices
$1-n$, $2-n$, and $3-n$. Thus $\beta$ is index-admissible.

Following the construction in \cite[Section~6.1]{GGM18},
we use $\beta$ to define the linearized contact homology
of $\alpha$; see also \cite{Par19,Cha25}.
This construction does not require $\alpha$ itself to be
index-admissible. The augmentation obtained from $\beta$
vanishes on non-contractible orbits, so the resulting complex
splits according to free homotopy classes. We take the summand generated
by the good contractible closed Reeb orbits of $\alpha$
and grade it by the Conley--Zehnder index, placing each
generator $\gamma$ in degree $\mu(\gamma)$.
Its homology is given by \eqref{eq:CH} or \eqref{eq:CH-nm};
see \cite[Section~6.2]{GGM18} and
Remark~\ref{rm:negative-grading}.
The relations \eqref{eq:resonance}--\eqref{eq:Morse_inequalities2}
apply when $\alpha$ is index-nonzero and has finitely many
simple contractible orbits, as in Section~\ref{sec:lesh}.
Hence the counting argument below also proves
Theorem~\ref{thm:main} under condition~\ref{cond:NFe}.

\begin{theorem}\label{thm:main'}
Let $(M^{2n+1},\xi)$, where $n$ is even, be a prequantization
$S^1$-bundle of a closed symplectic manifold $(B,\omega)$
such that $\omega|_{\pi_2(B)}\neq 0$ and $c_B>n/2$.
Assume that either $\H_k(B;\mathbb Q)=0$ for every odd $k$, or $c_B>n$. Let $\alpha$ be a non-degenerate contact form supporting~$\xi$ that has no contractible good periodic Reeb orbits $\gamma$ with Conley--Zehnder index $\mu(\gamma)=0$. Assume that one of the following conditions holds:
	\begin{itemize}
		\item[\reflb{cond:F1}{(i)}] $M$ admits a strong symplectic filling $(W,\Omega)$ such that $\Omega|_{\pi_2(W)}=0$ and $c_1(TW)|_{\pi_2(W)}=0$, and the inclusion-induced map $\pi_1(M) \to \pi_1(W)$ is injective. Furthermore, $\alpha$ is index-nonzero.
		\item[\reflb{cond:NF1}{(ii)}] $c_1(\xi)=0$, $B$ is spherically positive monotone, and $\alpha$ is index-admissible and index-positive.
	\end{itemize}
	Then $\alpha$ carries at least $\bar{r}_B$ geometrically distinct contractible periodic orbits, where
	\[
	\bar{r}_B:=\frac{1}{2}\chi(B) + \frac{1}{2}\dim \H_{n}(B;\Q)+1.
	\]
\end{theorem}

Let $(M,\xi)$ and $\alpha$ be as in Theorem~\ref{thm:main'}.
We begin the proof by assuming that $\alpha$ admits finitely many distinct simple contractible closed orbits $\gamma_1, \dots, \gamma_r$. If condition~\ref{cond:F1} in Theorem~\ref{thm:main'} is satisfied, then $B$ is spherically monotone by Proposition~\ref{prop:CH};
if  condition~\ref{cond:NF1} is satisfied, then $B$ is spherically positive monotone. We will focus on the case where $B$ is spherically positive monotone. When $B$ is spherically negative monotone, let $C_*$ be the complex generated by good contractible orbits and set $D_k=\operatorname{Hom}(C_{-k},\mathbb Q)$, with differential dual to $\partial:C_{-k+1}\to C_{-k}$. By \eqref{eq:range} and the finiteness of $\{\gamma_1,\ldots,\gamma_r\}$, the nonzero mean indices imply that each chain group is finite-dimensional. Formula~\eqref{eq:CH-nm} and Poincar\'e duality give
\[
\dim \H_k(D)=\sum_{m\geq1}\dim \H_{k-2mc_B+n}(B;\mathbb Q).
\]
Apply the argument below to $D_*$ and the inverse symplectic paths: the indices change sign, while good iterates remain good and the assumption $\mu(\gamma)\neq0$ for good contractible orbits is preserved. The same reasoning applies over the Novikov field in Theorem~\ref{thm:main}.

In the spherically positive monotone case, there is at least
one simple contractible closed orbit with positive mean index.
Indeed, otherwise all $\hmu(\gamma_i)$ would be negative,
and \eqref{eq:range} would imply
$\mu(\gamma_i^k)<n$ for every $i$ and every $k\geq1$.
Hence $c_m=0$ for every $m\geq n$, contradicting
\eqref{eq:CH} and \eqref{eq:Morse_inequalities1}, since
$\H_0(B;\mathbb Q)\ne0$.
After reordering the orbits, we therefore have:

\textbf{\reflb{MMI}{(MMI)}} {\it There exists an integer $r_0$ with $1 \leq r_0 \leq r$ such that $\hat{\mu}(\gamma_i) > 0$ for $i = 1, \dots, r_0$ and $\hat{\mu}(\gamma_i) < 0$ for $i = r_0+1, \dots, r$.}

Our goal is to establish the lower bound on $r$ stated in the theorem. Define
\begin{eqnarray}\label{l_0:definition}
&&\ell_0' = \max_{1 \le i \le r_0} \left\{ \min\left\{ k_0 \in \mathbb{N} \,\middle|\, \mu(\gamma_i^{k+\ell}) \ge \mu(\gamma_i^{k})+2n+1 \text{ for all } k \ge 1 \text{ and } \ell \ge k_0 \right\} \right\}, \nn\\
&&\ell_0'' = \max_{r_0+1 \le i \le r} \left\{ \min\left\{ k_0 \in \mathbb{N} \,\middle|\, \mu(\gamma_i^{k+\ell}) \le \mu(\gamma_i^{k})-2n-1 \text{ for all } k \ge 1 \text{ and } \ell \ge k_0 \right\} \right\}, \nn\\
&&\ell_0 = \max\{\ell_0', \ell_0''\}.
\end{eqnarray}
When $r_0=r$, we set $\ell_0^{\prime\prime}=0$. Otherwise, both $\ell_0'$ and $\ell_0''$ are well-defined and finite due to condition \ref{MMI}.

By Theorem \ref{indexrecurrence}, for any given $N \in \mathbb{N}$, $\eta > 0$, and $\ell_0$ as defined above, there exists a sequence of integer vectors $(d_j, k_{1j}, \dots, k_{rj})$ satisfying conditions \ref{cond:i} and \ref{cond:ii}, with all $d_j, k_{1j}, \dots, k_{rj}$ divisible by $N$. In fact, we only need one such vector from this sequence with $d_j, k_{1j}, \dots, k_{rj}$  sufficiently large. Therefore, we choose $j_0$ sufficiently large and set
\[
(d, k_1, \dots, k_r) := (d_{j_0}, k_{1j_0}, \dots, k_{rj_0}).
\]

The following lemma, which provides an expression for the truncated mean Euler characteristic (see Section \ref{sec:lesh}), is one of the key steps in the proof; cf. \cite[Sublemma 5.2]{Abr17} and \cite[Lemma 5.1]{GGM18}.

\begin{lemma}\label{lemma:resonance}
	The parameters $N$ and $\eta$ can be chosen so that $d = 2s c_B$ for some integer $s$, and
	\begin{equation*}
	\sum_{i=1}^r \sum_{k=1}^{k_i} \chi(\gamma_i^k) = \sum_{i=1}^r k_i \hat{\chi}(\gamma_i) = d \chi_+(M) = s \chi(B).
	\end{equation*}
\end{lemma}

\begin{proof}
Let $N$ be any positive integer multiple of $2c_B$. The first equality follows from the strong non-degeneracy of $\ga_1,\dots,\ga_r$ since the numbers $k_i$ are even. 
	
From \eqref{eq:CH}, it is easy to see that
\[
\chi_+(M) = (-1)^n \frac{\chi (B)}{2c_B},
\]
which implies the third equality due to $n$ being even. 
	
To prove the second equality, take $\eta$ sufficiently small such that $\displaystyle\eta\sum_{i=1}^{r}\frac{|\hat\chi(\gamma_i)|}{|\hat\mu(\gamma_i)|}<1$. By \eqref{eq:CH} and the assumption $c_B>n/2$, we have $b_{-m}=0$ for all $m\in\mathbb{N}\cup\{0\}$, so $\chi_-(M)=0$. Together with \ref{MMI} and the resonance relation \eqref{eq:resonance}, we conclude that 
\begin{equation}
\sum_{i=1}^{r} \frac{\hat\chi (\gamma_i)}{|\hmu (\gamma_i)|}=\sum_{i=1}^{r_0} \frac{\hat\chi (\gamma_i)}{\hmu (\gamma_i)}-\sum_{i=r_0+1}^{r} \frac{\hat\chi (\gamma_i)}{\hmu (\gamma_i)}=\chi_+(M).\label{Req}
\end{equation}
	
Then we have
\begin{align*}
d \chi_+(M) 
&= \sum_{i=1}^{r} \frac{d\hat\chi (\gamma_i)}{|\hmu (\gamma_i)|} \\
&= \sum_{i=1}^r  k_i\hat\chi (\gamma_i) +  
\sum_{i=1}^r  \frac{(d-k_i|\hmu(\ga_i)|)\hat\chi(\gamma_i)}{|\hmu(\ga_i)|} \\
&= \sum_{i=1}^r  k_i\hat\chi (\gamma_i) +  
\sum_{i=1}^r  \frac{(d-|\hmu(\ga_i^{k_i})|)\hat\chi (\gamma_i)}{|\hmu(\ga_i)|}.
\end{align*}

By condition (i) of Theorem \ref{indexrecurrence} and \eqref{Req}, we have 
\[
\bigg| \sum_{i=1}^r \frac{(d-|\hmu(\ga_i^{k_i})|)\hat\chi(\gamma_i)}{|\hmu(\ga_i)|}\bigg|
\leq \eta\sum_{i=1}^r\frac{|\hat\chi(\gamma_i)|}{|\hmu(\ga_i)|}<1.
\]
	
Note that by our choice of $N$, the numbers $d\chi_+(M)$ and $k_i\hat\chi (\gamma_i)$ for all $i$ are integers. Therefore,
\[
d \chi_+(M) =  \sum_{i=1}^r  k_i\hat\chi (\gamma_i).
\]
\end{proof}

Fix $N$ and $\eta$ as in Lemma \ref{lemma:resonance}. (In particular, $N$ is even.) By choosing $\eta$ sufficiently small,  we can ensure that the vector $(d, k_1, \dots, k_r) = (d_{j_0}, k_{1j_0}, \dots, k_{rj_0})$ with $j_0$ sufficiently large, as given by Theorem \ref{indexrecurrence}, satisfies the following relations for every $1 \leq i \leq r$ and $1 \leq \ell \leq \ell_0$:

\begin{equation}\label{eq:IRT1}
\cz(\ga_i^{k_i-\ell}) = \rho_i d - \cz(\ga_i^{\ell}),
\end{equation}
\begin{equation}\label{eq:IRT2}
\cz(\ga_i^{k_i+\ell}) = \rho_i d + \cz(\ga_i^\ell),
\end{equation}
and
\begin{equation}\label{eq:IRT3}
\lvert\cz(\ga_i^{k_i}) - \rho_i d\rvert \leq n
\end{equation}
(Note that since $N$ is even, the integers $d$ and $k_i$ are also even.)

To proceed, we introduce the following counting notation:
\begin{equation}\label{eq:r-e}
\begin{split}
r_{+}^{e}:=&\# \{ 1\leq i\leq r_0\mid \mu (\gamma_{i}^{k_{i}} ) \geq d,\ \gamma_i^{k_i}\text{ is good},\ \mu(\gamma_i^{k_i})\equiv 0\pmod{2} \}\\
&+\# \{r_0+1\leq i\leq r\mid \mu (\gamma_{i}^{k_{i}} ) \leq -d,\ \gamma_i^{k_i}\text{ is good},\ \mu(\gamma_i^{k_i})\equiv 0\pmod{2} \}
\end{split}
\end{equation}
and 
\begin{equation}\label{eq:r-o}
\begin{split}
r_{+}^{o}:=&\# \{ 1\leq i\leq r_0\mid \mu (\gamma_{i}^{k_{i}} ) \geq d,\ \gamma_i^{k_i}\text{ is good},\ \mu(\gamma_i^{k_i})\equiv 1\pmod{2} \}\\
&+\# \{r_0+1\leq i\leq r\mid \mu (\gamma_{i}^{k_{i}} ) \leq -d,\ \gamma_i^{k_i}\text{ is good},\ \mu(\gamma_i^{k_i})\equiv 1\pmod{2} \}.
\end{split}
\end{equation}
	
We now state the following key estimate:
\begin{lemma}\label{lemma:estimation}
	The following inequality holds:
	\begin{equation}
	r_{+}^{e}-r_{+}^{o} \geq \frac{1}{2} \chi(B) + \frac{1}{2} \dim \H_n(B; \mathbb{Q}),
	\end{equation}
\end{lemma}
\begin{proof}	
Observe that for each periodic orbit
$\ga_i$, there are four types of iterates outside $\ga_i^{k_i}$:
\begin{itemize}
	\item[(A)] $\ga_i^{k_i-\ell}$ with $\ell_0<\ell<k_i$;
	\item[(B)] $\ga_i^{k_i-\ell}$ with $1\leq \ell \leq \ell_0$;
	\item[(C)] $\ga_i^{k_i+\ell}$ with $1\leq \ell \leq \ell_0$;
	\item[(D)] $\ga_i^{k_i+\ell}$ with $\ell>\ell_0$.
\end{itemize}

We now analyze the contributions of these iterates to the alternating sum
\begin{equation} \label{eq:c_m}
\sum_{m=-d+1}^{d-1} (-1)^m c_m,
\end{equation}
where $c_m$ is the number of good contractible closed Reeb orbits of index $m$; see \eqref{lesh:nd}.

First, class (A) iterates have indices in $(-d+1, d-1)$ and hence all good orbits here contribute to \eqref{eq:c_m}.
Indeed, by definition of $\ell_0$, for all $\ell_0<\ell<k_i$:
\begin{align*}
\cz(\gamma_i^{k_i}) &\geq \cz(\gamma_i^{k_i - \ell}) + 2n+1 \quad \text{for } 1 \leq i \leq r_0, \\
\cz(\gamma_i^{k_i}) &\leq \cz(\gamma_i^{k_i - \ell}) - 2n-1 \quad \text{for } r_0+1 \leq i \leq r.
\end{align*}
Combined with \eqref{eq:IRT3}, this implies for all $\ell_0<\ell<k_i$:
\begin{equation}\label{classA}
\begin{aligned}
\cz(\gamma_i^{k_i - \ell}) &\leq d - n - 1 \quad \text{for } 1 \leq i \leq r_0, \\
\cz(\gamma_i^{k_i - \ell}) &\geq -d + n + 1 \quad \text{for } r_0+1 \leq i \leq r.
\end{aligned}
\end{equation}

By \eqref{eq:range} and \ref{MMI}, for every $k\geq1$ we have
\begin{equation}\label{eq:iterate-sign-bounds}
\begin{cases}
\mu(\gamma_i^k)>-n, & 1\leq i\leq r_0,\\
\mu(\gamma_i^k)<n, & r_0+1\leq i\leq r.
\end{cases}
\end{equation}
Choose $d$ sufficiently large that
\begin{equation}\label{eq:d-large}
 d>n+1,
 \qquad
 d>1+\max\bigl\{|\mu(\gamma_i^\ell)|:
 1\leq i\leq r,\ 1\leq\ell\leq\ell_0\bigr\}.
\end{equation}
Then \eqref{classA}, \eqref{eq:iterate-sign-bounds}, and \eqref{eq:d-large} imply that all iterates in class (A) have indices in $(-d+1,d-1)$. Moreover, \eqref{eq:IRT1}, \eqref{eq:IRT2}, and \eqref{eq:d-large} give, for $1\leq\ell\leq\ell_0$,
\begin{equation}\label{eq:nearby-iterate-bounds}
\begin{cases}
\mu(\gamma_i^{k_i\pm\ell})>-d+1, & 1\leq i\leq r_0,\\
\mu(\gamma_i^{k_i\pm\ell})<d-1, & r_0+1\leq i\leq r.
\end{cases}
\end{equation}

Class (D) orbits do not contribute to \eqref{eq:c_m} because for all $\ell > \ell_0$:
\begin{align*}
\cz(\gamma_i^{k_i + \ell}) &\geq \cz(\gamma_i^{k_i}) + 2n+1 \geq d + n+1 \quad \text{for } 1 \leq i \leq r_0, \\
\cz(\gamma_i^{k_i + \ell}) &\leq \cz(\gamma_i^{k_i}) - 2n-1 \leq -d - n-1 \quad \text{for } r_0+1 \leq i \leq r,
\end{align*}
where the second inequality in each case follows from \eqref{eq:IRT3}{.}

Since $k_i$ is even, \eqref{same parity} implies that
$\gamma_i^{k_i-\ell}$, $\gamma_i^{k_i+\ell}$, and
$\gamma_i^\ell$ are either all good or all bad, for
$1\leq\ell\leq\ell_0$.
Classes (B) and (C) can be further classified as follows:

\begin{itemize}
	\item[(B1)] $\gamma_i^{k_i-\ell}$ with $1 \leq \ell \leq \ell_0$, $\cz(\gamma_i^\ell) \ge 0$ for $1 \le i \le r_0$, or $\cz(\gamma_i^\ell) \le 0$ for $r_0+1 \le i \le r$;
	\item[(B2)] $\gamma_i^{k_i-\ell}$ with $1 \leq \ell \leq \ell_0$, $\cz(\gamma_i^\ell) < 0$ for $1 \le i \le r_0$, or $\cz(\gamma_i^\ell) > 0$ for $r_0+1 \le i \le r$;
	\item[(C1)] $\gamma_i^{k_i+\ell}$ with $1 \leq \ell \leq \ell_0$, $\cz(\gamma_i^\ell) \ge 0$ for $1 \le i \le r_0$, or $\cz(\gamma_i^\ell) \le 0$ for $r_0+1 \le i \le r$;
	\item[(C2)] $\gamma_i^{k_i+\ell}$ with $1 \leq \ell \leq \ell_0$, $\cz(\gamma_i^\ell) < 0$ for $1 \le i \le r_0$, or $\cz(\gamma_i^\ell)> 0$ for $r_0+1 \le i \le r$;
\end{itemize}

By \eqref{eq:IRT1} and \eqref{eq:nearby-iterate-bounds}, all good orbits from class (B1) contribute to \eqref{eq:c_m}.
The key point here is that $\alpha$ admits no good contractible periodic orbits of index zero (this is the main point where this condition is used).
For every good iterate $\gamma_i^{k_i-\ell}$ in class (B1), we have $\rho_i \cz(\gamma_i^\ell) > 0$, and by \eqref{eq:IRT1},
$
\cz(\gamma_i^{k_i-\ell}) = \rho_i d - \cz(\gamma_i^\ell),
$
so $\cz(\gamma_i^{k_i-\ell}) < d$ or $> -d$ whenever $\gamma_i^{k_i-\ell}$ is good.
In contrast, class (B2) makes no contribution to \eqref{eq:c_m}, since $\rho_i \cz(\gamma_i^\ell) < 0$, and hence
$
\cz(\gamma_i^{k_i-\ell}) = \rho_i d - \cz(\gamma_i^\ell) > d \text{ or } < -d
$
by \eqref{eq:IRT1}.

Similarly, \eqref{eq:IRT2} and \eqref{eq:nearby-iterate-bounds} show that class (C1) makes no contribution to \eqref{eq:c_m}, whereas all good orbits in class (C2) contribute to \eqref{eq:c_m}.
Indeed, by \eqref{eq:IRT2},
$
\cz(\gamma_i^{k_i+\ell}) = \rho_i d + \cz(\gamma_i^\ell),
$
which is $\ge d$ or $\le -d$ whenever $\rho_i \cz(\gamma_i^\ell) \ge 0$, and $< d$ or $> -d$ whenever $\rho_i \cz(\gamma_i^\ell) < 0$.

(Above, $\rho_i$ is defined by Theorem \ref{indexrecurrence}, and $\cz(\gamma_i^{k_i+\ell})$ and $\cz(\gamma_i^\ell)$ have the same parity since the integers $k_i$ are even.)

We set 
\begin{equation*} 
\begin{aligned}
c^e_\pm = &\sum_{i=1}^{r_{0}} \#\{1 \leq \ell \leq \ell_0 \mid \cz(\gamma_i^\ell) < 0,\ 
\gamma_i^{k_i\pm \ell}\text{ is good and }\cz(\gamma_i^\ell)\text{ is even}\} \\
&+ \sum_{i=r_{0}+1}^{r} \#\{1 \leq \ell \leq \ell_0 \mid \cz(\gamma_i^\ell) > 0,\ 
\gamma_i^{k_i\pm \ell}\text{ is good and }\cz(\gamma_i^\ell)\text{ is even}\}
\end{aligned}
\end{equation*}
and
\begin{equation*} 
\begin{aligned}
c^o_\pm = &\sum_{i=1}^{r_{0}} \#\{1 \leq \ell \leq \ell_0 \mid \cz(\gamma_i^\ell) < 0,\ 
\gamma_i^{k_i\pm \ell}\text{ is good and }\cz(\gamma_i^\ell)\text{ is odd}\} \\
&+ \sum_{i=r_{0}+1}^{r} \#\{1 \leq \ell \leq \ell_0 \mid \cz(\gamma_i^\ell) > 0,\ 
\gamma_i^{k_i\pm \ell}\text{ is good and }\cz(\gamma_i^\ell)\text{ is odd}\}.
\end{aligned}
\end{equation*}
Note that, since for all $i = 1, 2, \ldots, r$, $\cz(\gamma_i^{k_i-\ell})$ and $\cz(\gamma_i^{k_i+\ell})$ have the same parity by \eqref{same parity}, we have $c^e_+ = c^e_-$ and $c^o_+ = c^o_-$. Here by \eqref{lesh:nd} we remark that $c_\pm^e$ only counts the contributions of $\ga_i^{k_i\pm \ell}$ with $1\le \ell\le \ell_0$ to the even-th Morse-type numbers, and $c_\pm^o$ only counts the contributions of $\ga_i^{k_i\pm \ell}$ with $1\le \ell\le \ell_0$ to the odd-th Morse-type numbers.

Since $\ga_1,\dots,\ga_r$ are strongly non-degenerate, Lemma~\ref{lemma:resonance} and the preceding discussion give
\begin{equation}\label{cm1}
\begin{split}
\sum_{m=-d+1}^{d-1} (-1)^m c_m
&=\sum_{m=-d+1}^{d-1} (-1)^m \sum_{i=1}^{r}\sum_{k\geq 1}\dim \HC_{m}(\gamma_i^k)\\
&=\sum_{i=1}^{r} \sum_{\substack{k\geq 1\\ -d+1\leq\mu (\gamma_{i}^k)\leq d-1}} \chi(\gamma_{i}^k)\\
&=\sum_{i=1}^{r} \sum_{k=1}^{k_{i}} \chi(\gamma_{i}^k)
+ \sum_{i=1}^{r} \sum_{\substack{k\ge k_{i}+1\\ -d+1\leq\mu (\gamma_{i}^k)\leq d-1}} \chi(\gamma_{i}^k)\\
&\quad
-\sum_{i=1}^{r} \sum_{\substack{1\le k\le k_i-1\\
		\mu(\gamma_i^k)\geq d\text{ or }\mu(\gamma_i^k)\leq -d}} \chi(\gamma_{i}^k)
-\sum_{\substack{1\leq i\leq r\\
		\mu(\gamma_{i}^{k_{i}})\geq d\text{ or }\mu (\gamma_{i}^{k_i})\leq -d}}\chi(\gamma_{i}^{k_{i}})\\
&=s\chi(B)+(c_+^e-c_+^o)-(c_-^e-c_-^o)-(r_+^{e} - r_+^{o})\\
&=s\chi(B)-(r_+^{e} - r_+^{o}),
\end{split}
\end{equation}
where $r_+^{e}$ and $r_+^{o}$ are defined by \eqref{eq:r-e} and \eqref{eq:r-o}, respectively.

Next, we calculate the Betti number sum as follows. First, by \eqref{eq:CH} in Proposition \ref{prop:CH}, we establish the recurrence relation:
\begin{equation} \label{bm+2cb}
\begin{split}
b_{m+2c_B} &= \sum_{j \geq 1} \dim \H_{m+2c_B-2jc_B+n}(B;\mathbb{Q}) \\
&= \sum_{h \geq 0} \dim \H_{m-2hc_B+n}(B;\mathbb{Q}) \\
&= \dim \H_{m+n}(B;\mathbb{Q}) + b_m.
\end{split}
\end{equation}

Since $2c_B>n$, formula~\eqref{eq:CH} gives
$b_m=0$ for every $m\le0$ and
\[
b_m=\dim\H_{m-2c_B+n}(B;\mathbb Q),
\qquad 0\le m<2c_B.
\]
Moreover,
$\H_{m+2c_B+n}(B;\mathbb Q)=0$ for every $m\ge0$,
since $\dim B=2n$.
Consequently, \eqref{bm+2cb} yields
\[
b_{m+2jc_B}
=
b_m+\dim\H_{m+n}(B;\mathbb Q),
\qquad m\ge0,\quad j\ge1.
\]
Using these identities and $d=2sc_B$, we obtain
\begin{equation}\label{bettinumber}
\begin{split}
\sum_{m=-d+1}^{d-1} (-1)^m b_m &= \sum_{m=0}^{d-1} (-1)^m b_m \\
&= \sum_{m=0}^{2c_B-1} (-1)^m (b_m + b_{m+2c_B} + \cdots + b_{m+(2s-2)c_B}) \\
&= \sum_{m=0}^{2c_B-1} (-1)^m \left( s b_m + (s-1) \dim \H_{m+n}(B;\mathbb{Q}) \right) \\
&= s \sum_{m=0}^{2c_B-1} (-1)^m b_m + (s-1) \sum_{m=0}^{2c_B-1} (-1)^m \dim \H_{m+n}(B;\mathbb{Q}) \\
&= s \sum_{m=0}^{2c_B-1} (-1)^m \left( \dim \H_{m-2c_B+n}(B;\mathbb{Q}) + \dim \H_{m+n}(B;\mathbb{Q}) \right) \\
&\quad - \sum_{m=0}^{2c_B-1} (-1)^m \dim \H_{m+n}(B;\mathbb{Q}) \\
&= s \chi(B) - \sum_{m=0}^{n} (-1)^m \dim \H_{m+n}(B;\mathbb{Q}).
\end{split}
\end{equation}
By \eqref{cm1}, \eqref{bettinumber}, and the Morse inequality \eqref{eq:Morse_inequalities2}, we derive the following key relation:
\begin{equation}\label{eq:estimation'}
\begin{aligned} 
s\chi(B) - (r_+^{e} - r_+^{o}) 
&= \sum_{m=-d+1}^{d-1} (-1)^m c_m \\
&\le \sum_{m=-d+1}^{d-1} (-1)^m b_m \\
&= s\chi(B) - \sum_{m=0}^{n} (-1)^m \dim \H_{m+n}(B;\mathbb{Q}).
\end{aligned}
\end{equation}

Since $B$ is a $2n$-dimensional closed symplectic manifold, Poincar\'e duality implies $\H_{k}(B;\mathbb{Q}) \cong \H_{2n-k}(B;\mathbb{Q})$ for all $k \in \mathbb{Z}$. Consequently, \eqref{eq:estimation'} implies that
\begin{equation*} 
r_{+}^{e} - r_{+}^{o} \geq \frac{1}{2}\chi(B) + \frac{1}{2}\dim \H_{n}(B;\mathbb{Q}).
\end{equation*}
\end{proof}

\begin{proof}[Proof of Theorem~\ref{thm:main'}]
	Since $B$ is connected, $\dim\H_0(B;\mathbb Q)=1$.
	By Poincar\'e duality and the fact that $n$ is even,
    \[
	\frac12\chi(B)+\frac12\dim\H_n(B;\mathbb Q)
	=\sum_{j=0}^{n}(-1)^j\dim\H_j(B;\mathbb Q)
	\in\mathbb Z.
	\]
	By Lemma~\ref{lemma:estimation},
	\[
	r\geq r_+^e
	\geq r_+^e-r_+^o
	\geq
	\frac12\chi(B)+\frac12\dim\H_n(B;\mathbb Q).
	\]
	
	Suppose that
    \[
	r=
	\frac12\chi(B)+\frac12\dim\H_n(B;\mathbb Q).
	\]
	Then equality holds throughout the preceding chain
	of inequalities. In particular, $r=r_+^e$.
	By \eqref{eq:r-e}, each $\gamma_i^{k_i}$ is good
	and has even Conley--Zehnder index.
	All good iterates of the same underlying simple orbit
	have the same index parity.
	Hence every good contractible periodic orbit has even
	Conley--Zehnder index.
These orbits generate $\mathcal C_*^0(\alpha)$, so its chain groups
vanish in odd degrees. It follows that
\[
\HC_j^0(M)=0 \qquad\text{for every odd }j.
\]
Fix an odd integer $q$. In the spherically positive monotone case,
\eqref{eq:CH} identifies $\H_q(B;\mathbb Q)$ with the $m=1$ summand
of $\HC_{q+2c_B-n}^0(M)$. In the spherically negative monotone case,
\eqref{eq:CH-nm} identifies $\H_q(B;\mathbb Q)$ with the $m=1$
summand of $\HC_{q-2c_B-n}^0(M)$.
Both degrees are odd because $n$ is even. Since $\HC_j^0(M)=0$
for every odd $j$, it follows in either case that
\[
\H_q(B;\mathbb Q)=0
\qquad\text{for every odd }q.
\]
The same argument applies with Novikov coefficients, since
extension of scalars preserves the Betti numbers of $B$.
	
Every good contractible periodic orbit has even index,
	and index zero is excluded by assumption. Hence
\[
	\mu(\gamma)\notin\{-1,0,1\}
	\qquad
	\text{for every good contractible periodic orbit }\gamma.
	\]
	The hypotheses of
	Theorem~\ref{prop:previous-prequantization}
	are therefore satisfied. Since $n$ is even and
	$\H_{n-1}(B;\mathbb Q)=0$, its conclusion gives
\[
	r\geq \chi(B)+4\dim\H_{n-1}(B;\mathbb Q)=\chi(B).
	\]
On the other hand, Poincar\'e duality and the vanishing
	of the odd-degree homology established above imply
\[
	\chi(B)
	-\left(
	\frac12\chi(B)+\frac12\dim\H_n(B;\mathbb Q)
	\right)
	=
	\sum_{j=0}^{n-1}\dim\H_j(B;\mathbb Q)
	\geq 1.
	\]
Therefore,
\begin{equation}\label{eq:contra}
    r\geq\chi(B)>
    \frac12\chi(B)+\frac12\dim\H_n(B;\mathbb Q)=r,
\end{equation}
which is impossible. Thus
$r>\frac12\chi(B)+\frac12\dim\H_n(B;\mathbb Q)$.
Since both sides are integers, we conclude that
\[
    r\geq\frac12\chi(B)+\frac12\dim\H_n(B;\mathbb Q)+1.
\]
This proves the theorem.
\end{proof}

\section{Proof of Theorem \ref{thm:minor}}

Let $(M^{5},\xi)$ be the prequantization $S^1$-bundle over a closed symplectic manifold $(B,\omega)$ such that
 $\omega|_{\pi_2(B)} \neq 0${,} $c_B > 1$, and
$\H_{1}(B;\mathbb{Q}) = 0$.
Let $\alpha$ be a non-degenerate contact form supporting $\xi$ that has no simple contractible periodic orbits $\gamma$ with $\cz(\gamma)=0$. Assume that $M$ satisfies either condition~\ref{cond:F0} or condition~\ref{cond:NF0} in Theorem~\ref{thm:minor}.

We begin the proof by assuming that $\alpha$ admits finitely many distinct simple contractible closed orbits $\gamma_1, \dots, \gamma_r$.

Since $\gamma_1, \dots, \gamma_r$ are
strongly  non-degenerate,
by \eqref{eq:CZ_iteration} in Theorem~\ref{indexformula}, we have for each $\gamma=\gamma_i$
\begin{eqnarray}
&&\mu(\ga^k)
=k\left(\mu(\ga)-r_i \right) + 2\sum_{j=1}^{r_i}\left\lfloor\frac{k\theta_{i,j}}{2\pi}\right\rfloor + r_i  \label{index-iterate},\\
&&\hat\mu(\ga)=\mu(\ga)-r_i+\sum_{j=1}^{r_i}\frac{\theta_{i,j}}{\pi}\label{meanindex-iterate},
\end{eqnarray}
where $r_i\in\{0,1,2\}$ and $\frac{\theta_{i,j}}{2\pi}\in(0,1)\setminus \mathbb{Q}$ for $1\le j\le r_i$.

\begin{lemma}\label{lemma:nonzeroindex}
For any $ 1\le i\le r $, the Conley--Zehnder index of each good $ k $-th iterate $ \ga_{i}^{k} $ of $\ga_{i} $ with $ k \in \mathbb{N} $ satisfies $ \cz(\ga_{i}^{k})\neq 0 $.
\end{lemma}

\begin{proof}
We argue by contradiction. Suppose that there exist an integer $1 \leq i_0 \leq r$ and an integer $k_0 \geq 2$ such that the iterate $\gamma_{i_0}^{k_0}$ is good and $\cz(\gamma_{i_0}^{k_0}) = 0$. A bad orbit has only bad iterates, so $\gamma_{i_0}$ is good. Its index therefore has the same parity as that of $\gamma_{i_0}^{k_0}$. 
By the hypothesis that $\cz(\gamma_{i_0}) \neq 0$, it follows that $\mu(\gamma_{i_0}) \in 2\mathbb{Z} \setminus \{0\}$, and hence
\begin{equation}\label{eq:est}
|\mu(\gamma_{i_0})| \geq 2.
\end{equation}
Using \eqref{index-iterate}, we obtain
\begin{equation}\label{eq:est1}
\cz(\gamma_{i_0}^{k_0}) = k_0(\mu(\gamma_{i_0}) - r_{i_0}) + 2\sum_{j=1}^{r_{i_0}} \left\lfloor \frac{k_0\theta_{i_0,j}}{2\pi} \right\rfloor + r_{i_0} = 0.
\end{equation}
Combining \eqref{eq:est} and \eqref{eq:est1}, we have
\begin{equation}\label{eq:est2}
\left| (k_0-1) r_{i_0} - 2\sum_{j=1}^{r_{i_0}} \left\lfloor \frac{k_0\theta_{i_0,j}}{2\pi} \right\rfloor \right| = k_0 |\mu(\gamma_{i_0})| \geq 2k_0.
\end{equation}
Note that for $j = 1, \dots, r_{i_0}$, we have $\frac{\theta_{i_0,j}}{2\pi} \in (0,1) \setminus \mathbb{Q}$ and $0 \leq \left\lfloor \frac{k_0\theta_{i_0,j}}{2\pi} \right\rfloor \leq k_0-1$. Therefore,
\begin{equation}\label{eq:est3}
\left| (k_0-1) r_{i_0} - 2\sum_{j=1}^{r_{i_0}} \left\lfloor \frac{k_0\theta_{i_0,j}}{2\pi} \right\rfloor \right| \leq (k_0-1) r_{i_0}.
\end{equation}
Combining \eqref{eq:est2} and \eqref{eq:est3}, we obtain $r_{i_0} \geq \frac{2k_0}{k_0-1} > 2$, which contradicts the fact that $r_{i_0} \leq 2$. The proof is complete.
\end{proof}

\begin{lemma}\label{lemma:nonzeromean}
	For any $1 \le i \le r$, we have $\hmu(\gamma_i) \neq 0$.
\end{lemma}
\begin{proof}
	We proceed by contradiction. Assume that there exists an integer $1\leq i_1\leq r$ such that $\hmu(\gamma_{i_1}) = 0$. Then, by equation \eqref{meanindex-iterate}, we have
	\begin{equation}\label{eq:i1}
	\hmu(\gamma_{i_1}) = \mu(\gamma_{i_1}) - r_{i_1} + \sum_{j=1}^{r_{i_1}} \frac{\theta_{i_1,j}}{\pi} = 0,
	\end{equation}
	where $r_{i_1}\in\{0,1,2\}$ and $r_{i_1} \neq 1$ because $\frac{\theta_{i_1,j}}{2\pi} \notin \mathbb{Q}$ for $1 \le j \le r_{i_1}$.
	
	We now consider two cases:
	
	\textbf{Case 1:} If $r_{i_1} = 0$, then \eqref{eq:i1} implies $\mu(\gamma_{i_1}) = 0$, which contradicts the hypothesis of Theorem~\ref{thm:minor}.
	
	\textbf{Case 2:} If $r_{i_1} = 2$, then from \eqref{eq:i1} we have
	\[
	\mu(\gamma_{i_1}) = 2 - \sum_{j=1}^2 \frac{\theta_{i_1,j}}{\pi} \in (-2, 2),
	\]
	so $\mu(\gamma_{i_1}) \in \{-1, 0, 1\}$.
	Since $r_{i_1} = 2$, Proposition~\ref{index:odd-even} imply that $\mu(\gamma_{i_1}) \in 2\mathbb{Z}$. Therefore, $\mu(\gamma_{i_1}) = 0$, which again contradicts the hypothesis of Theorem~\ref{thm:minor}.
	
	In both cases, we reach a contradiction. The proof is complete.
\end{proof}

\begin{proof}[Proof of Theorem \ref{thm:minor}]
Since $\H_1(B;\mathbb Q)=0$, Poincar\'e duality gives
\[
\H_3(B;\mathbb Q)=0,
\qquad
\dim\H_4(B;\mathbb Q)=\dim\H_0(B;\mathbb Q){{}=1}.
\]
Hence $\chi(B)={2}+\dim\H_2(B;\mathbb Q)$. 
	
By Lemma~\ref{lemma:nonzeromean}, Lemma~\ref{lemma:nonzeroindex}, and all the assumptions in Theorem~\ref{thm:minor}, we see that the conditions of Theorem~\ref{thm:main} for $ n=2 $ are all satisfied. Therefore, by Theorem~\ref{thm:main}, the number of geometrically distinct contractible periodic orbits carried by $\alpha$ is at least 
$$
\frac{1}{2}\chi(B) + \frac{1}{2} \dim \H_{2}(B;\mathbb{Q})+{1}=\dim \H_{\ast}(B;\mathbb{Q}).
$$
Thus, the proof is complete.
\end{proof}

\section*{Acknowledgements}
The author is supported by the Postdoctoral Fellowship Program of CPSF under Grant Number GZC20252019.

\end{document}